\documentclass[12pt,leqno]{article}

\usepackage{fancyhdr}
\usepackage[english]{babel}
\usepackage{amsmath,amsfonts,amssymb,amsthm}
\usepackage{graphicx}
\usepackage{amscd}
\usepackage{url}
\usepackage{vmargin}
\usepackage{mathrsfs}
\usepackage{color}

\definecolor{Green}{rgb}{0.20,0.43,0.09}

\usepackage{pdfpages}
\usepackage[bookmarksopen=false,breaklinks=true,  backref=page,pagebackref=true,plainpages=false,%
    hyperindex=true,pdfstartview=FitH,colorlinks=true, pdfpagelabels=true,colorlinks=true,linkcolor=blue,citecolor=red]%
   {hyperref}
\hypersetup{
    colorlinks=true,
    linkcolor=blue,
    citecolor=red,
    filecolor=green,
    urlcolor=blue,
}

\newtheorem{theorem}{Theorem}

\newtheorem{corollary}[theorem]{Corollary}

\newtheorem{definition}[theorem]{Definition}
\newtheorem{temp-def}{Temporary definition}

\newtheorem{example}{Example}

\newtheorem{lemma}[theorem]{Lemma}
\newtheorem{lemma-def}[theorem]{Lemma-Definition}
\newtheorem{sublemma}[theorem]{Sublemma}

\newtheorem{proposition}[theorem]{Proposition}

\newtheorem{remark}[theorem]{Remark}

\def\proof{\smallskip\goodbreak{\it Proof.~~--~\kern.3em}
     \ignorespaces}%
\def\qedbox{$\square$}%
\def\qed{\ifmmode\qedbox\else\unskip\ \hglue0mm\hfill
     \qedbox\smallskip\goodbreak\fi}%

\def\proofof#1{\smallskip\goodbreak{\it Proof of #1.~~--~\kern.3em}
     \ignorespaces}%
\def\qedbox{$\square$}%
\def\qed{\ifmmode\qedbox\else\unskip\ \hglue0mm\hfill
     \qedbox\smallskip\goodbreak\fi}%
\newcommand{\ZZ}{\mathbb{Z}}

\newcommand{\CC}{\mathbb{C}}

\newcommand{\NN}{\mathbb{N}}

\newcommand{\C}{\mathcal{C}}
\newcommand{\D}{\mathcal{D}}
\newcommand{\E}{\mathcal{E}}
\newcommand{\F}{\mathcal{F}}
\newcommand{\G}{\mathcal{G}}
\renewcommand{\H}{\mathcal{H}}

\renewcommand{\O}{\mathcal{O}}

\newcommand{\Cc}{\mathscr{C}}

\newcommand{\Ec}{\mathscr{E}}
\newcommand{\Fc}{\mathscr{F}}
\newcommand{\Gc}{\mathscr{G}}

\newcommand{\Lc}{\mathscr{L}}

 \newcommand{\CA}{\mathcal{A}}

\def\sym{\texttt{sym}}
\def\Sym{\mathrm{Sym}}

\begin{document}

\newcommand\shorttitle{Galoisian Methods for the Irreducibility of ODE}
\newcommand\authors{G. Casale and J.-A. Weil}
\title{Galoisian Methods for Testing Irreducibility of Order Two Nonlinear Differential Equations}

\author{G. Casale\thanks{IRMAR, Universit\'e de Rennes 1, 35042 RENNES Cedex, France \texttt{guy.casale@univ-rennes1.fr} } 
\& J.-A. Weil\thanks{XLIM,  Universit\'e de Limoges, 123 avenue Albert Thomas, 87060 Limoges Cedex, France \texttt{weil@unilim.fr} }
}

\date{\today }

\maketitle

\begin{abstract}
  The aim of this article is to provide a method to prove the irreducibility of non-linear ordinary differential equations by means of the differential Galois group of their variational equations along algebraic solutions. We show that if the dimension of the Galois group of a variational equation is large enough then the equation must be irreducible. 
 We propose a method to compute this dimension via reduced forms. 
 As an application, we reprove the irreducibility of the second and third Painlev\'e equations for special values of their parameter.
 In the Appendix, we recast the various notions of variational equations found in the literature and prove their equivalences.
\end{abstract}

\tableofcontents

\noindent\textbf{Keywords:} Ordinary Differential Equations, Differential Galois Theory, Painlev\'e Equations, Computer Algebra.

\noindent\textbf{MSC 2010 Classification:} 
34M55, 
34M03, 
34A05, 
34M15, 
34A26, 
34M25, 
20G05, 
17B45.  

\section*{Introduction}

To study differential equations by reduction to a simpler form is as old as differential equations themselves. 
However, the first formalized definition of reducibility appeared only in the  Stockholm 
lessons of Paul Painlev\'e \cite{painleve}.
A complete algebraization of this definition was given by K. Nishioka \cite{nishioka} and H. Umemura \cite{umemura}. Note that Nishioka's concept of decomposable extension may be more general than reducibility.
The first application was the proof the irreducibility of the first Painlev\'e equation \cite{painleve-irred,nishioka,umemura, umemura2}. H. Umemura gave a simple criterion to prove irreducibility and the Japanese school applied it to all Painlev\'e equations \cite{noumi-okamoto, umemura-watanabe,umemura-watanabe-bis,watanabe, watanabe2}. 
\\
These papers deal with reducibility of \emph{solutions}; in this paper, we will emphasize on the (stronger) notion of \emph{reducibility of an equation} (see next section for proper definitions).

Painlev\'e suggested (see \cite{painleve-note}) that irreducibility of a differential equation can be proved by the computation of its (hypothetical)
 ``rationality group'', as (incorrectly) defined by J. Drach in \cite{drach-these}. 
Such a group-like object was finally defined by H.Umemura \cite{umemura-galois} (where it is a group functor) and B. Malgrange \cite{malgrange} (where it is an algebraic pseudogroup), see also \cite{pommaret}. 
In \cite{casale-irred}, G. Casale showed that  some properties of this Malgrange pseudogroup imply the irreducibility of the differential equation. 
Casale applied this  in \cite{casale-P1} to prove the irreducibility of the first Painlev\'e equation. S. Cantat and F. Loray used this  in \cite{CA}
to prove the irreducibility of the sixth Painlev\'e equation.

The computation of the Malgrange pseudogroup of a differential equation is a difficult (and currently wide open) problem. 
In this paper,  we use differential Galois groups
 of the variational equations along an algebraic solution of equations of the form $y''=f(x,y)$ to determine their Malgrange pseudogroup and then prove their irreducibility. 

The study of an equation through its linearization is ancient. Applications to integrability of differential equations were greatly 
improved by S.L. Ziglin \cite{ziglin} followed by many authors, notably J.J. Morales-Ruiz and J.-P. Ramis \cite{MR1,MR2} and then with C. Sim\`o \cite{MRS} using the differential Galois group of the variational equations along a solution. In \cite{casale-MR}, Casale proved that these Galois groups provide a lower bound for the Malgrange pseudogroup in the following way. 
This pseudogroup acts on the phase space and the algebraic solution (along which we linearize) parameterizes a curve $\Cc$ in this space. Then the group of $k$-jets of its elements fixing a point  in $\Cc$ contains the Galois group of the $k$-th order variational equation along $\Cc$.\\

Using technics developed in \cite{MR1, MR2,MRS} and the Malgrange pseudogroup following \cite{casale-MR}, we will prove the following theorem:

\begin{theorem}\label{main-theorem}
  Let $M$ be a smooth irreducible algebraic $3$-fold over $\CC$ and $X$ be a rational vector field on $M$ such that there exist a closed 
  rational $1$-form $\alpha$ with $\alpha(X)=1$ and a closed rational 2-form $\gamma$ with $\iota_{X}\gamma = 0$.

  Assume $\Cc$ is an algebraic $X$-invariant curve with $X_\Cc \not \equiv 0$. If the Galois group of the first variational equation of $X$ 
  along $\Cc$ is not virtually solvable and the dimension of the Galois group of the formal variational equation is greater than $5$ then 
  the Malgrange groupo\"{\i}d is
  $$Mal(X) = \{\varphi \ | \ \varphi^\ast \alpha = \alpha, \varphi^\ast\gamma = \gamma \}.$$ 
 Moreover, if there exist rational coordinates $x,y,z$ on $M$ such that $X = \frac{\partial}{\partial x} + z \frac{\partial}{\partial y} +f(x,y,z) \frac{\partial}{\partial z}$ then the equation $y'' = f(x,y,y')$ is irreducible.
\end{theorem}  

Casale proved in \cite{casale-irred} that 
the second point in the conclusion of this theorem is a direct consequence of the first. 
Another way to express the conclusion of the theorem is that the singular holomorphic foliation $\Fc_{X}$ of $M$ defined by trajectories of $X$ has no transversal rational geometric structure except the transversal rational volume form given by $\gamma$.

This theorem can be applied to prove the irreducibility of  equations of the form $y'' = f(x,y)$. Solutions $x \mapsto (x,y(x),y'(x))$ of such an equation are trajectories of the vector field $\frac{\partial}{\partial x}+ z\frac{\partial}{\partial y}+ f(x,y)\frac{\partial}{\partial z}$ on the phase space. The forms $\alpha = dx$ and $\gamma = \iota_X(dx \wedge dy \wedge dz)$ are closed and $\alpha(X)=1$, $\iota_X\gamma = 0$. To apply the theorem, a particular solution is needed. This is the case in our main example. The second Painlev\'e equation with parameter $a$ is 
$$
(P_2) \quad y'' = xy + 2y^3 +a .
$$
When $a=0$, it admits the algebraic solution $y=0$.
M. Noumi and K. Okamoto prove in \cite{noumi-okamoto} that, apart from the trivial solution $y=0$, the solutions of this equation are irreducible in the sense of Nishioka-Umemura. The approach presented here
uses the Malgrange pseudogroup of the rational vector field $X$ on $M =\CC^3$ given by
$$
X = \frac{\partial}{\partial x} + z\frac{\partial}{\partial y} +(xy+2y^3)\frac{\partial}{\partial z}, 
$$
whose trajectories are parameterized by solutions of $(P_2)$. Using notation of the theorem, $\alpha = dx$, $\gamma = \iota_X(dx\wedge dy\wedge dz)$ and $\Cc = \{y=z=0\}$, we prove that
$$
Mal(X) =\{\varphi \ | \ \varphi^\ast \alpha = \alpha, \varphi^\ast\gamma = \gamma \}.
$$
This equality implies the irreducibility of $(P_2)$. Note that this property of the Malgrange pseudogroup is  much stronger than irreducibility in Nishioka-Umemura sense. However, it is not a   completely algebraic property:  it is formulated for the differential field $\CC(x)$ and seems to be specific to differential fields which are finitely generated over the constants, whereas the definition of irreducibility can be stated over any differential field.\\

The application of our theorem to prove the irreducibility of the second Painlev\'e equation requires two steps.
 
 First, one needs to check whether the Galois group of the first variational equation is solvable by finite (or virtually solvable). This differential equation reduces to the Airy equation $y''=xy$
and it is easy, for example by using the Kovacic algorithm  \cite{kovacic}, to show that its differential Galois group is $SL(2,\CC)$.

Then, to check the dimension condition seems more hasardous at first sight. We would need to compute Galois group of higher order variational equations until we would find a Galois  group of dimension greater than $6$. 
Until now, no bound is known on the order of the required variational equation that one would have to study to prove this. Moreover, the size of the (linearized) variational equations grows fast and, even though there are theoretical methods to compute differential Galois groups in \cite{hrushovski},
the computation of differential Galois groups of such big systems is yet unrealistic in general.
\\
In our case, the situation is better because the methods of P.H. Berman and M.F. Singer \cite{berman-singer,berman} could allow us to determine the differential Galois group. We choose another approach, following the works of A. Aparicio and J.-A. Weil on reduced forms of linear differential systems (see \cite{AJA,AEJA}), notably \cite{AJA2,AJA3} where new effective technics allow to compute the Lie algebra of the differential Galois group of a variational equation of order $k$ when 
the variational equation of order $k-1$ has an abelian differential Galois group.
We show how to extend their method to our situation.

These computations can be reused to prove irreducibility of a larger class of differential equations: $y'' = xy + y^nP(x,y)$. We will then show how this technique can be used to prove the irreducibility of a family of Painlev\'e III equations.
\\

The paper is organized as follows. Section $1$ contains the definitions of reducibility,  variational equations and its differential Galois group in 
order to state the main theorem. 
In Section 2, we elaborate a simple irreducibility criterion for equations of the form $y'' = xy + y^nP(x,y)$ and give two irreducibility proofs for a Painlev\'e II equation. In Section 3, we apply a similar scheme to prove the irreducibility of a Painlev\'e III equation from statistical physics.\\
In the appendices, we detail the constructions and prove the main theorem. 
In Appendix A, we recast the Galois groups in the context of $G$-principal connection. 
In Appendix B,  we describe and compare various notions of variational equations (arc 
space and frame bundle viewpoints), as the literature is occasionally hazy on this point. In appendix C, we recall the definition of the Malgrange pseudogroup of a vector field and give some of its properties regarding 
the reducibility and the variational equations. Together  with the Cartan classification of pseudogroups in dimension $2$ (in a neighborhood of a generic point), this allows us to finally prove our main theorem 1.

\section{Definitions}

\subsection{Irreducibility}

In the 21st lesson of his Stockholm's lessons \cite{painleve}, P. Painlev\'e defined different classes of transcendental functions and gave the definition of order two differential equation reducible to order one. Then he proved that the so-called Picard-Painlev\'e equation, a special case of Painlev\'e sixth equation discovered by E. Picard, is irreducible. This proof relies on the fact that this equation has no moving singularities and then its flow gives bimeromorphic transformations of the plane $\CC^2$. In this situation, reducible equations have a flow sending a foliation by algebraic curves onto another algebraic one. This is not the case for the case for Picard-Painlev\'e equation. 

Later, Painlev\'e claimed without proof that the computation of Drach's rationality group from \cite{drach-these} would prove the irreducibility of an equation.
He tried to compute it for the first Painlev\'e equation in \cite{painleve-note}.   

\begin{definition}[\cite{painleve,nishioka,umemura}]\label{reducible-solution}
  Let $(K,\delta)$ be an ordinary differential field, $y$ a differential indeterminate and $(E): \delta^2 y = F(y,\delta y) \in K(y,\delta y)$ a second order differential equation defined on $K$.
  A solution of the equation $(E)$ is called a \emph{reducible solution} if it lies in 
  a differential extension $L$ of $K$ built in the following way:
    $$
    K=K_{0} \subset K_{1} \subset \ldots \subset K_{m}=L 
    $$
  with one of the following elementary extension for any $i$:
    \begin{itemize}
      \item either $K_{i} \subset K_{i+1}$ is an algebraic extension,
      \item or $K_{i} \subset K_{i+1}$ is a linear extension,
             {\it i.e.} $K_{i+1} = K_{i}(f^p_{j}; 1 \leq p,j \leq n )$ with $\delta f^p_{j} = \sum_{k} A^k_{j} f^p_{k}$, $ A^k_{j} \in K_{i}$.
      \item or $K_{i} \subset K_{i+1}$ is an abelian extension,
             {\it i.e.}  $K_{i+1} = K_{i}(\varphi_{j}(a_{1},\ldots,a_{n});1 \leq j \leq n)$ with $\varphi$'s a basis of periodic functions on $\CC^n$ given by the field of rational functions
             on an abelian variety over $\CC$ and  $a$'s in $K_{i}$, 
      \item or $K_{i} \subset K_{i+1}$ has transcendence degree $1$,
             {\it i.e.} $K_{i+1} = K_{i}(z, \delta z)$ with $P(z,\delta z) = 0$, $P \in K_{i}[X,Y]-\{0\}$.
    \end{itemize}
\end{definition}

Note that Nishioka's definition of decomposable extension seems more general that reducibility. We don't know any example of a decomposable irreducible extension nor any proof of the equivalence of the two notions. In the articles of Umemura, the notion of reducible appears together with the notion of classical functions. The latter is similar except that the last kind of elementary extension is not allowed.

This definition may not be the most relevant to understand the geometry of the differential equation:
an order two differential equation may have
  two functionnally independent first integrals in a Picard-Vessiot extension of $\CC(x,y,z)$ without being reducible.
This is the case for Picard-Painlev\'e equation as it is explained in the 21st lesson of Painlev\'e \cite{painleve}, see also \cite{casale-PP} and \cite{watanabe2}.  

The above definition is a property of individual solutions; however, the equation may have an exceptional solution which is reducible whereas the others are not. For example, any  equation $\delta^2 y = y F(y, \delta y) + \delta y G(y, \delta y) \in K[y,\delta y]$ admits $y=0$ as solution. Therefore we will introduce a notion of \emph{reducibility of the equation} which translates, in algebraic terms, the idea  that the general solution of the equation is reducible.

\begin{definition}\label{reducible-equation}
  Let $(K,\delta)$ be an ordinary differential field and $(E): \delta^2 y = F(y,\delta y) \in K(y,\delta y)$ a second order differential equation defined on $K$.
  The equation $(E)$ is called a \emph{reducible differential equation over $K$} 
  if there exists a reducible solution $f$ such that $transc.deg.(K(f ,\delta f)/K) = 2 $ (i.e. the general solution of the equation is reducible).
\end{definition}

\begin{example}
Consider the equation $\delta^2 y =0$. We want to show that it is reducible over $(\CC(x), \delta = \frac{\partial}{\partial x})$. Its general solution is $f=a x + b$ for arbitrary (i.e transcendental) constants $a$ and $b$. 
Here, $K=\CC(x)$ and $K(f ,\delta f) = \CC(a,b)(x)$ (with $a$ and $b$ transcendental over $\CC$) so that we indeed have
$transc.deg.(K(f ,\delta f)/K) = 2 $. This is why, in the second condition for reducibility of solutions (in definition \ref{reducible-solution} above), we allow \emph{linear} extensions with possibly new constants (and not only Picard-Vessiot extensions).
\end{example}

\begin{remark} Note that it does not seem fully clear whether "all" solutions of  a reducible equation are reducible.
\end{remark}



Using the Malgrange pseudogroup of a vector field and \'E.~Cartan's classification of pseudogroups, Casale proved in \cite{casale-P1}  the following theorem.

\begin{theorem}[Annexe A in \cite{casale-P1}]
 Let $X$ be a rational vector field on $M$, a smooth irreducible algebraic $3$-fold. Assume there exist a rational closed $1$-form $\alpha$ such that $\alpha(X) = 1$ and a rational closed $2$-form $\gamma$ such that $\iota_{X}\gamma = 0$. 
 Then one of the three following holds.
    \begin{itemize}
      \item There exists a $1$-form $\omega$ with coefficients in $\overline{\CC(M)}^{alg}$ such that $\omega(X)=0$ and 
              for any local determination of algebraic functions $\omega \wedge d\omega = 0$.
      \item There exist $\theta_1$, $\theta_2$ two rational $1$-form vanishing on $X$ and $(\theta_{i}^j)$ a traceless $2\times 2$ matrix
              of rational $1$-form such that $\theta_i(X) = 0$, $d\theta_i = \sum _{k} \theta_i^k\wedge \theta_k$ and $d\theta_i^j = \sum_{k}\theta^j_{k} \wedge\theta^k_{i}$, $\forall (i,j) \in \{1,2\}^2$.
              \item The Malgrange pseudogroup is $Mal(X) =\{\varphi | \varphi^\ast \alpha = \alpha;\varphi^\ast (\gamma) = \gamma  \}$. 
    \end{itemize}
\end{theorem}

The systems of PDE given in the first two items of the statement are the analogue of the resolvant equations in classical Galois theory. 
The existence of a rational solution to the resolvant equations would  imply that the Malgrange pseudogroup is small.  
Then in \cite{casale-irred}, the claim of Painlev\'e is proved.

\begin{theorem}[\cite{casale-irred}]
 Let $E$ be a rational order two equation $y'' = F(x,y) \in \CC(x,y)$ and 
  $X = \frac{\partial}{\partial x} + z\frac{\partial}{\partial y} + F(x,y)\frac{\partial}{\partial z}$ be the rational vector field on $\CC^3$ 
  associated to $E$. If $Mal(X) =\{\varphi | \varphi^\ast dx = dx;\varphi^\ast (\iota_X dx\wedge dy \wedge dz) = \iota_X dx\wedge dy \wedge dz  \}$ then $E$ is irreducible. 
\end{theorem}

\subsection{Variational Equations}

Let $X$ be a vector field on an algebraic manifold $M$ and $\Cc \subset M$ an algebraic $X$-invariant curve such that 
$X_{\Cc} \not \equiv 0$. Variational equations can be written easily in local coordinates. Intrinsic versions will be given in  the appendix. 
In local coordinates $(x_1, \ldots, x_n)$ on $M$, the flow equations of $X = \sum a_i(x)\frac{\partial}{\partial x_i}$ are 
$$ 
\frac{d}{dt} x_i  = a_i(x) \quad i= 1,\ldots,n.
$$
This flow can be used to move germs of analytic curves on $M$ pointwise. Let $\epsilon \mapsto x(\epsilon)$ be such a germ defined on $(\CC,0)$.
For any $\epsilon$ small enough, one has 
$$ 
\frac{d}{dt} x_i(\epsilon) = a_i(x(\epsilon)) \quad i= 1,\ldots,n.
$$
Analyticity allow us to expand this equality. Let $x(\epsilon) = \displaystyle \left(\sum_k x_1^{(k)} \frac{\epsilon^k}{k!}, \ldots,\sum_k x_n^{(k)} \frac{\epsilon^k}{k!} \right)$
then
\begin{equation}
\label{variational system}
\tag{$VE_{k}$} 
\left\{
  \begin{array}{l}
	\displaystyle \frac{d}{dt} x_i^0 = a_i(x^0) \\
	\displaystyle		\frac{d}{dt} x_i^{(1)} = \sum_j \frac{\partial a_i}{\partial x_j}(x^0) x_j^{(1)} \\
	\displaystyle \frac{d}{dt} x_i^{(2)} = \sum_j \frac{\partial a_i}{\partial x_j}(x^0) x_j^{(2)} + \sum_{j,\ell} \frac{\partial^2 a_i}{\partial x_j \partial x_\ell}(x^0) x_j^{(1)}x_\ell^{(1)}\\
%
%
	\displaystyle \frac{d}{dt} x_i^{(3)} = \sum_j \frac{\partial a_i}{\partial x_j}(x^0) x_j^{(3)} + \sum_{j,\ell} 3 \frac{\partial^2 a_i}{\partial x_j \partial x_\ell}(x^0) x_j^{(2)}x_\ell^{(1)}  \\
	\hspace{7cm}\displaystyle  + \sum_{j,\ell,m} \frac{\partial^3 a_i}{\partial x_j \partial x_\ell \partial x_m}(x^0) x_j^{(1)}x_\ell^{(1)}x_m^{(1)}\\
	\vdots \\
	\displaystyle \frac{d}{dt} x_i^{(k)} = F_k(\partial^\beta a_i (x_0) , x_i^{(\ell)}\ |\ i=1,\ldots, n; |\beta|\leq k; \ell \leq k )
  \end{array}
\right.
\end{equation}

where the $F$'s are given by Faa di Bruno formulas (see formula (14) p. 860 in \cite{MRS}). The order $k$ variational equation is the differential system on order $k$ jets of 
parameterized curve on $M$ obtained in this way. Because $\Cc$ is an algebraic $X$-invariant curve, the space of parameterized curves with $x^{0} \in \Cc$ is an algebraic subvariety invariant by the variational equation. The variational equation gives a non-linear connexion on the bundle over $\Cc$ of parameterized curves pointed on $\Cc$. This restriction is the variational equation along $\Cc$.

The system (\ref{variational system}) is a rank $nk$ non-linear system but it can be linearized. For instance the 
order $3$ variational equation is linearized using new unknowns  $z_{\ell,k,j} = x_\ell^{(1)}x_k^{(1)}x_j^{(1)}$, $z_{k,j}=x_k^{(2)}x_j^{(1)}$ and $z_k = x^{(3)}_k$, which amounts to perform some tensor constructions on lower order linearized variational equations (such as symmetric powers of the first variational equation), see \cite{simon, AJA2, MRS}. 
The linear system obtained is
\begin{equation}
\label{linear variational system}
\tag{$LVE_{3}$}
\left\{
\begin{array}{l}
\displaystyle \frac{d}{dt} x_i^0 = a_i(x^0) \\
\\
\displaystyle
\frac{d}{dt} z_{\ell,k,j} = \sum_{b,c,d} \left(\frac{\partial a_\ell}{\partial x_b} + \frac{\partial a_k}{\partial x_c} + \frac{\partial a_j}{\partial x_d}\right)(x^0) z_{b,c,d}\\
\\
\displaystyle \frac{d}{dt} z_{k,j} = \sum_{b,c} \left(\frac{\partial a_\ell}{\partial x_b} + \frac{\partial a_k}{\partial x_c}\right)(x^0) z_{b,c} + \sum_{c,d} \frac{\partial^2 a_k}{\partial x_c \partial x_d}(x^0) z_{c,d,j}\\
\\
%
\displaystyle \frac{d}{dt} z_i = \sum_j \frac{\partial a_i}{\partial x_j}(x^0) z_j + \sum_{j,\ell} 3 \frac{\partial^2 a_i}{\partial x_j \partial x_\ell}(x^0) z_{j,\ell} + \sum_{j,\ell,m} \frac{\partial^3 a_i}{\partial x_j \partial x_\ell \partial x_m}(x^0) z_{j,\ell,m}
\end{array}\right.
\end{equation} 
 
 When $X$ preserves a transversal fibration $\pi: M \to B$, the parameterized curves $\epsilon \to x(\epsilon)$ included in fibers of $\pi$ give a subset of curves invariant by $X$. The restriction of the variational equation to this subset is called the $\pi$-normal variational equation. The main case of interest is the normal variational equation of an ODE. Such a differential equation gives a vector field $\frac{\partial}{\partial x_{1}} +\ldots$ where $x_{1}$ is the independent coordinate.  The normal variational equation (with respect to the projection on the curve of the independent coordinate) is obtained from the variational equation by setting $x_{1}^{(k)} = 0$ when $k \geq 1$. 
 
 The order $k$ linearized normal variational equation is obtained from the order $k$ linearized variational equation by setting $z_{\alpha} = 0$ when a coordinate of $\alpha \in \NN^k$ is equal to $1$. The induced system will be denoted by $NLVE_{k}$

 \subsection{The Galois Group and the Main Theorem}
 
 Following E. Picard and E. Vessiot, the differential Galois group of a linear differential system $\frac{d}{dt} Y = A Y$ with $A \in GL(n,\CC(t))$ can be defined in the following way. 
 
 First, select a regular point $t_{0}$ of the differential system and a fundamental matrix $F(t)\in GL(\CC\{t-t_{0}\})$ of holomorphic solutions at this point. Then the splitting field, called Picard-Vessiot extension, is $L = \CC(t, F_{i}^j(t) | 1\leq i,j \leq n)$ and the differential Galois group $G$ is the group of $\CC(t)$-automorphisms of $L$ commuting with $\frac{d}{dt}$. 
 
 Picard proved that this group $G$ is an linear algebraic subgroup of $GL(n,\CC)$ and Vessiot proved the Galois correspondence.  
 In our context, the linearized normal variational equation is a subsystem of the linearized variational equation so the Galois correspondence implies that its Galois group is a quotient of the Galois group of the variational equation. 

Introductions to this theory may be found in \cite{magid} or the reference book \cite{put-singer}. 
Other variations on that theme can be found e.g. in 
\cite{katz, kolchin, bertrand}. We propose an overview of the theory from the ``principal bundle'' point of view in the appendices.

The statement of our main theorem involves the Malgrange pseudogroup of a vector field wich is a non linear generalisation of the differential Galois group; we recall its definition in Appendix C.2. 

\newcounter{theoreme1}
\setcounter{theoreme1}{\value{theorem}}
\setcounter{theorem}{0}

\begin{theorem}
  Let $M$ be a smooth irreducible algebraic $3$-fold over $\CC$ and $X$ be a rational vector field on $M$ such that there exist a closed 
  rational $1$-form $\alpha$ with $\alpha(X)=1$ and a closed rational 2-form $\gamma$ with $\iota_{X}\gamma = 0$.

  Assume $\Cc$ is an algebraic $X$-invariant curve with $X_\Cc \not \equiv 0$. If the Galois group of the first variational equation, $(VE_{1})$, of $X$ 
  along $\Cc$ is not virtually solvable and if the exists a $k \geq2$ such that the dimension of the Galois group of the $(LVE_{k})$ is greater than $5$ then the Malgrange pseudogroup is
  $$Mal(X) = \{\varphi \ | \ \varphi^\ast \alpha = \alpha, \varphi^\ast\gamma = \gamma \}.$$ 
 Moreover, if there exist rational coordinates $x,y,z$ on $M$ such that $X = \frac{\partial}{\partial x} + z \frac{\partial}{\partial y} +f(x,y,z) \frac{\partial}{\partial z}$ then the equation $y'' = f(x,y,y')$ is irreducible.
\end{theorem}  

\setcounter{theorem}{\value{theoreme1}}

In the application, we will compute the Galois group of the normal variational equation. As this group is a quotient of the group used in the theorem,  one can replace $(VE_{1})$ by $(NVE_{1})$ and $(LVE_{k})$ by $(NLVE_{k})$ without changing the conclusion of the theorem.
We postpone the proof of the theorem to the appendices because it requires additional technology which is recalled there. In the next two sections, we show applications of this theorem to the irreducibility of order two equations such as the Painlev\'e equations $(P_{II})$ and $P_{III}$.

\section{Irreducibility of $\frac{d^2y}{dx^2}=f(x,y)$ and the Painlev\'e II Equation}

We will compute the differential Galois group of some normal variational equation along the solution $y=0$ of differential equations of the form:
$$
\frac{d^2 y}{dx^2} = xy +y^nP(x,y) \text{ with } P \in \CC(x,y) \text{ without poles along }y=0\text{ \ and \ } n \geq 2.
$$ 

The vector field of our equation is
  $$
    X = \frac{\partial}{\partial x} + z\frac{\partial}{\partial y} + (xy + y^nP(x,y))\frac{\partial}{\partial z}.
  $$ 
This equation has a solution $y=z=0$. The first normal variational equation along this curve is 
$$
  \frac{\partial}{\partial x}  +   z^{(1)}\frac{\partial}{\partial y^{(1)}}  + x y^{(1)}\frac{\partial}{\partial z^{(1)}}
$$
Using a parameterisation $x=t$ of this curve, we  get a linear system:

$$ \frac{d}{dt}Y = A \; Y \quad \textrm{ with } \quad A=\left( \begin {array}{cc} 0&1\\ \noalign{\medskip}t&0\end {array}
 \right).$$
It is easily seen, from the form of the equation, that the variational equations of order less than $n$ bring no new information, 
because of the term in $y^n$. 
Letting $y=\sum_{i=1}^n y^{(i)} \frac{\epsilon^i}{i!}$ and $p(x) = n!P(x,0)$, we have 
	$$ x\; y \; + \; y^n P(x,y) = \sum_{i=1}^{n-1} x y^{(i)} \epsilon^i + \left( x y^{(n)}+ (y^{(1)})^n p(x)\right) \epsilon^n + o(\epsilon^n)$$
and the $n$-th order normal variational equation along the solution $y=z=0$ is
  $$
    \frac{\partial}{\partial x}  + \left( \sum_{k=1}^{n-1} z^{(k)}\frac{\partial}{\partial y^{(k)}} + x y^{(k)}\frac{\partial}{\partial z^{(k)}} \right)+ z^{(n)}\frac{\partial}{\partial y^{(n)}} +  (x y^{(n)} + p(x) ({y^{(1)}})^n)\frac{\partial}{\partial z^{(n)}} 
  $$
The linearized  normal variational system can be reduced to 
  $$
(NVE_{n}): \qquad   \frac{d}{dt}
     \left(\begin{matrix}
       \vdots \\ \vdots \\
       {n \choose k} (y^{(1)})^{n-k} (z^{(1)})^{k}\\
       \vdots \\
       \hline \\
       y^{(n)} \\
       z^{(n)} 
     \end{matrix}\right)
    =  
     \left( 
       \begin{array}{ccc|cc}
         & & & \multicolumn{2}{c}{\vdots}\\
         & \sym^n{\left(\begin{matrix}0&1\\ t & 0 \end{matrix}\right)} & &  \multicolumn{2}{c}{0} \\
         & & &  \multicolumn{2}{c}{\vdots}\\
         \hline
         0 & \dots &0& 0&1\\
         p(t)&\dots &0 & t&0 \\
       \end{array}
     \right)\cdot
    \left( \begin{matrix}
        \vdots \\ \vdots \\
         {n \choose k} (y^{(1)})^{n-k} (z^{(1)})^{k}\\
        \vdots \\
        \hline \\
        y^{(n)} \\
        z^{(n)} 
      \end{matrix}\right).
  $$ 

\begin{example}
For example, in the case of the second Painlev\'e equation with $a=0$, $(P_2)$, we have $n=3$ and the linearized variational system is 
$$ \frac{d}{dt} Y = \CA \cdot Y \quad \textrm{ with } \quad \CA= 
	\left( \begin {array}{cccc|cc} 
		0&1&0&0&0&0\\ 
		\noalign{\medskip}3\,t&0& 2&0&0&0\\ 
		\noalign{\medskip}0&2\,t&0&3&0&0\\ 
		\noalign{\medskip}0&0&t&0&0&0\\ \hline \\[-0.6cm]
		\noalign{\medskip}0&0&0&0&0&1\\ \noalign{\medskip} 12 &0&0&0&t&0
\end {array}
\right)   \qquad (PNVE_3).$$
\end{example}

\subsection{Reduced Forms and a First Irreducibility Proof of $(P_2)$}

We introduce material from \cite{AEJA,AJA} about the Kolchin-Kovacic reduced forms of linear differential systems.\\
Consider a differential system $[A]:\; Y' = AY$ with $A\in Mat(n,k)$. Let $G$ denote its differential Galois group and $\mathfrak{g}$ its Lie algebra.
Given a matrix $P\in GL(n,\overline{k})$, the change of variable $Y=P.Z$ transforms $[A]$ into a system $Z'=B.Z$, where $
B=PAP^{-1}-P'P^{-1}$. The standard notation is $B=P[A]$. The systems $[A]$ and $[P[A]]$ are called \emph{equivalent over $\overline{k}$}.
The Galois group may change but its Lie algebra $\mathfrak{g}$ is preserved under this transformation.
\\
We say that $[A]$ is \emph{in reduced form} if $A\in\mathfrak{g}(k)$. When it is not the case, we say that a matrix $B\in Mat(n,\overline{k})$ is a 
\emph{reduced form} of $[A]$ if there exists $P\in GL(n,\overline{k})$ such that $B=P[A]$ and  $B\in \mathfrak{g}(\overline{k})$. 
Our technique to find $\mathfrak{g}$, for the variational equations, will be to transform them into reduced form.

\begin{example} The first variational equation of Painlev\'e II has matrix $A_1=\left(\begin{matrix} 0&1\\t&0\end{matrix}\right)$. This corresponds to the Airy equation and its Galois group is known to be $SL(2,\CC)$. Obviously, $A_1 \in \mathfrak{sl}(2,\CC(t))$ so the first variational equation is in reduced form.
\end{example}

Let $a_1(x),\ldots, a_r(x)\in k$ be a basis of the  $\mathbb{C}$-vector space generated by the coefficients of $A$. We may decompose $A$ as 
$A=\sum_{i=1}^s a_i(x) M_i$, where the $M_i$ are constant matrices (and $s$ is minimal). The \emph{Lie algebra associated to $A$}, denoted $Lie(A)$ is the \emph{algebraic Lie algebra generated by the $M_i$}: it is the smallest Lie algebra which contains the $M_i$ and is also the Lie algebra of some (connected) linear algebraic group $H$, see \cite{AEJA}.

\begin{example} \label{painleve-example}
We compute the Lie algebra $Lie(\CA)$ associated to $\CA$ in system $(PNVE_3)$.
Let
$$ X:=  {\scriptsize{\left( \begin {array}{cccc|cc} 0&1&0&0&0&0\\ \noalign{\medskip}0&0&2&0
&0&0\\ \noalign{\medskip}0&0&0&3&0&0\\ \noalign{\medskip}0&0&0&0&0&0 \\ \hline \\[-0.6cm]
 \noalign{\medskip}0&0&0&0&0&1\\ \noalign{\medskip}0&0&0&0&0&0
\end {array}\right) }}
= \left( 
  \begin{array}{c|c}
     \sym^n\left(\begin{matrix} 0& 1\\ 0&0 \end{matrix}\right) & 0 \\
     \hline
     0 & \begin{matrix} 0& 1\\ 0&0 \end{matrix} \\
  \end{array}
\right),
   \; 
Y:=    {\scriptsize{\left( \begin {array}{cccc|cc} 0&0&0&0&0&0\\ \noalign{\medskip}3&0&0&0
&0&0\\ \noalign{\medskip}0&2&0&0&0&0\\ \noalign{\medskip}0&0&1&0&0&0 \\ \hline\\[-0.6cm] 
\noalign{\medskip}0&0&0&0&0&0\\ \noalign{\medskip}0&0&0&0&1&0
\end {array}\right) }}
= \left( 
    \begin{array}{c|c}
    \sym^n\left(\begin{matrix} 0& 0\\ 1&0 \end{matrix}\right) & 0 \\
    \hline
    0 & \begin{matrix} 0& 0\\ 1&0 \end{matrix} \\
  \end{array}
\right) , \; 
$$
and
$$  
H:=[X,Y]=  {\scriptsize{\left( \begin {array}{cccc|cc} 3&0&0&0&0&0\\ \noalign{\medskip}0&1&0&0
&0&0\\ \noalign{\medskip}0&0&-1&0&0&0\\ \noalign{\medskip}0&0&0&-3&0&0 \\ \hline\\[-0.6cm]
 \noalign{\medskip}0&0&0&0&1&0\\ \noalign{\medskip}0&0&0&0&0&-1
\end {array}\right)}}
$$
(standard $\mathfrak{sl}_2$ triplet with $[H,X]= 2\; X$, $[H,Y]= - 2\; Y$)
and we introduce the off-diagonal matrices 
$$ E_i = \left( 
       \begin{array}{ccc|c}
         \ddots & & & \multicolumn{1}{c}{\vdots}\\
         & 0 & &  \multicolumn{1}{c}{0} \\
         & & \ddots&  \multicolumn{1}{c}{\vdots}\\
         \hline
         & B_i & & 0\\
      
       \end{array}
     \right)
$$
where
$$B_0= {\scriptsize{\left( \begin {array}{cccc} 0&0&0&0\\ \noalign{\medskip}1&0&0&0\end {array}\right)}}
, 
B_1 ={\scriptsize{\left( \begin {array}{cccc} +1&0&0&0\\ \noalign{\medskip}0&-1&0&0\end {array}\right)}}
,
B_{{2}}= {\scriptsize{\left( \begin {array}{cccc} 0&-1&0&0\\ \noalign{\medskip}0&0&1&0\end {array}\right)}},
B_{{3}}= {\scriptsize{\left( \begin {array}{cccc} 0&0&+1&0\\ \noalign{\medskip}0&0&0&-1\end {array}\right)}},
B_{{4}}={\scriptsize{ \left( \begin {array}{cccc} 0&0&0&-1\\ \noalign{\medskip}0&0&0&0\end {array}\right)}}. 
$$
Moreover:  $[X,E_i]= (i+1) E_{i+1}$, $[Y,E_i]= (5-i) E_{i-1}$, $[H,E_i]=(-4+2i) E_i$ and $[E_i,E_j]=0$
 (with $E_{-1}=E_5=0$). 
 We now show that $Lie(\CA)$  is generated (as a Lie algebra) by $X$, $Y$ and $E_1$. 
 Indeed, $Lie(\CA)$ is generated (as a Lie algebra) by $M_1:=X+  E_1$ and $M_2:=Y$;
 then $[M_1,M_2]=H$ and $[M_1,H]=-2X-4E_1$ so $[M_1,H]+2 M_1 = -2 E_1$ and so $E_1\in Lie(\CA)$.
 
 The above calculations then shows that $Lie(\CA)$ has dimension 8 and a basis for it is $\{ X,Y,H, E_0,\ldots, E_4\}$. 
 We admit, this is proved later, that this Lie algebra is actually an algebraic Lie algebra.

%
 
 \end{example}
 
 
 A theorem of Kolchin  (\cite{put-singer}, proposition 1.31) shows that $\mathfrak{g}\subset Lie(A)$ (and that $G\subset H$). A reduced form is obtained when we achieve equality in that inclusion. Moreover, when $G$ is connected (which will be the case in this paper), the reduction theorem of Kolchin and Kovacic (\cite{put-singer}, corollaire 1.32) shows that a reduced form exists and that the reduction matrix $P$ may be chosen in $H(k)$.
\\

Let us now  continue the above examples with the third variational equation of Painlev\'e II.\\
Denote by $\mathfrak{h}_{diag}$ the Lie algebra generated by the block-diagonal elements $X,Y,H$.
Similarly, let $\mathfrak{h}_{sub}$ be the Lie algebra generated by the off-diagonal matrices $E_i$.
Of course, $\mathfrak{h}_{diag}$ is $\mathfrak{sl}_2$ in its representation on a direct sum $\Sym^n(\mathbb{C}^2)\oplus \mathbb{C}^2$.
\\
Letting $\mathfrak{h}:=Lie(A)$, we see that $\mathfrak{h} = \mathfrak{h}_{diag} \oplus \mathfrak{h}_{sub}$. It is easily seen that 
$\mathfrak{h}_{diag}$ is a Lie subalgebra of $\mathfrak{h}$ and that $\mathfrak{h}_{sub}$ is an ideal in $\mathfrak{h}$.
\\
 
We have seen that $\mathfrak{g} \subset Lie(A)$. Furthermore, $\mathfrak{h}_{diag} \subset \mathfrak{g}$ 
(because $VE_1$ has Galois group $SL_2(\mathbb{C})$). It follows that $\mathfrak{g} = \mathfrak{h}_{diag} \oplus \tilde{\mathfrak{g}}$, where $\tilde{\mathfrak{g}} \subset \mathfrak{h_{sub}}$ is an ideal in $\mathfrak{g}$; in particular, it is closed under the bracket with elements of $\mathfrak{h}_{diag}$ (adjoint action of $\mathfrak{h}_{diag}$ on $\mathfrak{h_{sub}}$). 
\\

Now the only invariant subsets of $\mathfrak{h}_{sub}$ under this adjoint action are seen to be $\{0\}$ and $\mathfrak{h}_{sub}$ (this is reproved and generalized in proposition \ref{irreducibility-prop} below and its lemmas).
So the Lie algebra $\mathfrak{g}$ is either $\mathfrak{sl}_2$ (of dimension 3) or $\mathfrak{h}$ (of dimension $8$).
\\
 
 As the Galois group of the block-diagonal part is connected, the differential Galois group $G$ of $[\mathcal{A}]$  is connected.
Hence we know (by the reduction theorem of Kolchin and Kovacic cited above)  that there exists a reduction matrix $P\in H(k)$. Furthermore,  as the block-diagonal part of $\CA$ is already in reduced form,  the block-diagonal part of the reduction matrix $P$ may be chosen to be the identity.  
 So there exists a reduction matrix  of the form 
 	$$P = \textrm{Id} + \sum_{i=1}^5 f_i(t) E_i, \; \textrm{ with }Ê\; f_i(t) \in \mathbb{C}(t).$$
 \\
 
 A simple calculation shows that $P\mathcal{A} P^{-1} = \mathcal{A} + \sum_{i=1}^5 f_i(t) [E_i, X+t Y]$ and hence 
	$$ P[\mathcal{A}] = X + t Y + E_1 + \sum_{i=1}^5 f_i(t) [E_i, X+t Y] - \sum_{i=1}^5 f_i'(t) E_i.$$ 
We see that the case 
 $\mathfrak{g}=\mathfrak{sl}_2$  happens if and only if we can find $f_i\in \CC(t)$ such that $\sum_{i=1}^5 f_i'(t) E_i = \sum_{i=1}^5 f_i(t) [ X+t Y, E_i] + E_1$.
 Let $\Psi$ denote the matrix of the adjoint action $[ X+t Y, \bullet]$ of $X+ t Y$ on $\mathfrak{h}_{sub}$. 
 We see that  $\mathfrak{g}=\mathfrak{sl}_2$ iff we can find an $F\in \CC(t)^5$  solution of the differential system 
 	$$ F' = \Psi\cdot F + \left(\begin{matrix} 1 \\ 0 \\\vdots \\ 0 \end{matrix}\right).$$
 We now gather the properties of $(P_2)$ elaborated in this sequence of examples.
 
 \begin{proposition} The Painlev\'e II equation is irreducible when the parameter $a = 0$. \end{proposition}
 \begin{proof} Using the Barkatou algorithm and its Maple implementation \cite{Ba99a, BaClElWe12a}, one easily sees that the above differential system does not have a rational solution. If follows that,using the notations of the above examples, we have $\mathfrak{g}=\mathfrak{h}$ of dimension 8. 
 So, for $(P_2)$, we have:  the Galois group of the first variational equation is $SL(2,\CC)$ which is not virtually solvable; the Galois group of the third variational equation has dimension $8>5$. Theorem 1 thus shows that the Painlev\'e II equation is irreducible. 
 \end{proof}

We will now generalize this process to all equations of the form $\frac{d^2 y}{dx^2} = xy +y^nP(x,y)$. We will elaborate a much easier irreducibility criterion, which will allow to reprove the above proposition without having to trust a computer.


\subsection{The Galois Group of the $n$-th Variational Equation} 

The aim of this subsection is to prove the following:
\begin{proposition}\label{irreducibility-prop}
  The Galois group of the $n$-th variational equation $(LNVE_{n})$ is either $SL_{2}(\CC)$ or its dimension is $n+5$ and then the differential equation $y'' = xy +y^nP(x,y)$ 
  is irreducible.
\end{proposition}

\subsubsection{Adjoint Action}
\begin{lemma}
Let $A$ be a $2 \times 2$ matrix of rational function of the variable $t$ such that the Galois group $G_1$ of the differential system
$
\frac{dY}{dt} = A Y
$
has Lie algebra $\mathfrak{sl}_{2}$.
Consider  a system
$$
\frac{d}{dt}
\left(\begin{matrix}
  Z\\
  Y 
\end{matrix}\right)
= 
\left(
\begin{array}{c|c}
  \sym^n A & 0 \\
 \hline
  B &A\\
\end{array}
\right)
\left(\begin{matrix}
  Z\\
  Y 
\end{matrix}\right)
$$
with differential Galois group $G$. 
Then $G$ has dimension $3$ or $n+3$ or $n+5$ or $2n+5$.
\end{lemma} 
\begin{proof}
Let $\mathfrak{g}$ be the Lie algebra of $G$. 
It has a block lower triangular form shaped by the form of the system {\it i.e.} if $\mathfrak{g} \subset  \mathfrak{h} \subset \mathfrak{gl}_{n+3}$ where
$$ \mathfrak{h} = \left\{
\left( 
\begin{array}{c|c}
  \sym^n a & 0 \\
 \hline
  b & a\\
\end{array}
\right)
,a \in \mathfrak{sl}_{2}, b\in (\mathbb{C}^2)^\vee \otimes \Sym^n(\CC^2)\right\}.$$
The south-east block $A$ defines a subsystem thus $G$ contains a subgroup isomorphic to $SL_{2}$.
The north-west block defines a quotient system so there is a surjective group morphism from  $G$ onto $SL_{2}$. 
The kernel of this map is a commutative ideal (the off-diagonal matrices $E_i$, in our examples) 
and inherits a structure of $\mathfrak{sl}_{2}$-module for the inclusion of $\mathfrak{sl}_{2}$ in $\mathfrak{h}$ {\it via} $\mathfrak{g}$. As a representation, 
$\mathfrak{g} \cap (\mathbb{C}^2)^\vee \otimes \Sym^n(\CC^2)$ is a subspace of $(\mathbb{C}^2)^\vee \otimes \Sym^n(\CC^2)$. This representation is nothing but the adjoint representation. The decomposition in irreducible representations is 
(see ex. 11.11 in \cite{fulton-harris}, part 11.2 p. 151, or next paragraph)
$$
(\mathbb{C}^2)^\vee \otimes \Sym^n(\CC^2) = \Sym^{n-1}(\CC^2) \oplus \Sym^{n+1}(\CC^2).
$$ 
%
So the Lie algebra $\mathfrak{g}$ is either $\mathfrak{sl}_{2}$, or $\mathfrak{sl}_{2} \rtimes \Sym^{n-1}(\CC^2)$, or $\mathfrak{sl}_{2} \rtimes \Sym^{n+1}(\CC^2)$
or $\mathfrak{sl}_{2} \rtimes \left(\Sym^{n-1}(\CC^2) \oplus \Sym^{n-1}(\CC^2)\right)$.
Its dimension is then $3$ or $n+3$ or $n+5$ or $2n+5$.
\end{proof}

\subsubsection{Vector Field Interpretation}
In order to compute easily, we will use the following identification.
The Lie algebra $\mathfrak{sl}_{2}$ may be viewed as a Lie algebra of linear vector fields on $\CC^2$ namely $\CC X + \CC H+ \CC Y$ with 
$X = x \frac{\partial}{\partial y}$, $H = x \frac{\partial}{\partial x} - y \frac{\partial}{\partial y}$ et $Y = y \frac{\partial}{\partial x}$.
These are the same standard $X$, $Y$ and $H$ as the matrices of Example \ref{painleve-example}.
\\
The dual representation $\mathbb{C}^2 \otimes \Sym^n((\CC^2)^\vee)$ is the space of vector fields on $\CC^2$ whose coefficients  are
homogeneous polynomial of degree $n$. 
The decomposition in irreducible representation is the decomposition of any vector field 
	in $\mathbb{C}^2 \otimes \Sym^n((\CC^2)^\vee)$ as 
$$ A\frac{\partial}{\partial x} + B\frac{\partial}{\partial y} = G(x,y) (x\frac{\partial}{\partial x} + y \frac{\partial}{\partial y}) + \frac{\partial K}{\partial y} \frac{\partial}{\partial x} 
- \frac{\partial K}{\partial x} \frac{\partial}{\partial y} $$ 
with\footnote{via $G= \frac{1}{n+1}\left( \frac{\partial A}{\partial x} + \frac{\partial B}{\partial y}\right)$
and $K = \frac{1}{n+1} \left( y A - x B \right)$.}
$G\in \Sym^{n-1}((\CC^2)^\vee)$ and $K \in \Sym^{n+1}((\CC^2)^\vee)$. 

The \emph{symplectic gradient} of a polynomial $K$ will be denoted by 
	$$J\nabla K:=\frac{\partial K}{\partial y} \frac{\partial}{\partial x} 
- \frac{\partial K}{\partial x} \frac{\partial}{\partial y}.$$
If we define
	$$ K_i:=\left( n+1 \choose i\right) x^{n+1-i} y^i 
	\; \textrm{ and } \; E_i:=\frac1{n+1} J\nabla(K_i),$$
then calculation shows that, just like in the previous section, 
	$$ [X, E_i] = (i+1) E_{i+1}, \quad [Y,E_i]=(n+2-i) E_{i-1} \quad \textrm{ and }\quad [H,E_i]=(2i-n-1)E_i.$$

\begin{lemma}
Let $\mathfrak{h}:=Lie(\mathcal{A})$ be the Lie algebra associated to the matrix $\mathcal{A}$ of $(LNVE_{n})$.
Let $G$ denote the differential Galois group of  $(LNVE_{n})$ and $\mathfrak{g}$ be its Lie algebra.
With the standard notations of Example \ref{painleve-example}, we have: 
\begin{enumerate}
	\item $\mathfrak{h}$ is generated, as a Lie algebra, by $X$, $Y$ and $E_0$ and $\mathfrak{h}=\mathfrak{sl}_{2} \rtimes \Sym^{n+1}(\CC^2)$.
	\item Either $\mathfrak{g}=\mathfrak{sl}(2)$ (of dimension 3) or $\mathfrak{g}=\mathfrak{h}$ (of dimension $n+5$).
\end{enumerate}
\end{lemma}

\begin{proof}
With the matrices of Example \ref{painleve-example}, we have $\mathcal{A} = X + t Y + p(t) E_0$. 
As $[X,E_i]=(n+1-i)E_{i+1}$, the Lie algebra generated by $X$, $Y$ and $E_0$
has dimension $n+5$ and may be identified with $\mathfrak{sl}_{2} \rtimes \Sym^{n+1}(\CC^2)$. 
Moreover, a Lie algebra containing $X$, $Y$ and any of the $E_i$ contains $\mathfrak{sl}_{2} \rtimes \Sym^{n+1}(\CC^2)$ (because $[Y,E_i]=(n+2-i) E_{i-1}$).

If $1$, $t$ and $p(t)$  are linearly independent over $\CC$  then  $Lie(\mathcal{A})$ is the algebraic envelop of the Lie algebra generated by $X$, $Y$ and $E_0$; as the latter is algebraic (it is $\mathfrak{sl}_{2} \rtimes \Sym^{n+1}(\CC^2)$), we have $Lie(\mathcal{A})=\mathfrak{sl}_{2} \rtimes \Sym^{n+1}(\CC^2)$. 
Now, we have seen that $\mathfrak{g}$ is either $\mathfrak{sl}_{2}$, or $\mathfrak{sl}_{2} \rtimes \Sym^{n-1}(\CC^2)$, or $\mathfrak{sl}_{2} \rtimes \Sym^{n+1}(\CC^2)$
or $\mathfrak{sl}_{2} \rtimes \left(\Sym^{n-1}(\CC^2) \oplus \Sym^{n-1}(\CC^2)\right)$. Among these, only $\mathfrak{sl}_{2}$ and $\mathfrak{sl}_{2} \rtimes \Sym^{n+1}(\CC^2)$ are in $Lie(\mathcal{A})$, which proves the lemma in that case.
\\

We are left with the case $p(t) = a+bt$ with $(a,b)\in \CC^2$. Then $Lie(\mathcal{A})$ is the algebraic envelop of the Lie algebra generated by
$M_1:=X + a E_0$ and $M_2:= Y + b E_0$.
If $b=0$, then $[M_1,M_2]=H$ and $[M_1,H]=2X + a(n+1) E_0$ so $[M_1,H]-2 M_1 = a (n-1) E_0$. So $Lie(\mathcal{A})$ contains $E_0$ and we are done.
If $b\neq 0$ then let $M_3:=[M_1,M_2]=H+(n+1)^2b E_1$; then $[M_3,Y]= - 2Y - (n+1)b E_0$ so $2 M_2- [M_3,Y]$ is a multiple of $E_0$ and the result is again true.
\end{proof}

\begin{proofof}{Proposition \ref{irreducibility-prop}} Follows from the above two lemmas and Theorem 1. \end{proofof}

\subsection{Irreducibility Criteria}

Thanks to Proposition \ref{irreducibility-prop}, to show irreducibility of $y'' = xy +y^n \frac{P(x,y)}{Q(x,y)}$, it is enough to show
that $(LNVE)_n$ has a Lie algebra not isomorphic to $\mathfrak{sl}(2)$. Using the Kolchin-Kovacic reduction theory, we achieve this by proving (as in our first proof of irreducibility of Painlev\'e II) that there is no reduction matrix that transforms our system to one with Lie algebra $\mathfrak{sl}(2)$.
This gives us the following simple irreducibility criterion.

\begin{proposition}\label{irreducibility-criterion-airy}
 We consider the equation $(E):\; y''=x y + y^n P(x,y)$. Let $p(t):=n!P(t,0)$.
Let $L_{n+1}:=\sym^{n+1}(\partial_t^2-t)$ denote the $(n+1)$-th symmetric power of the Airy equation.
If the equation $L_{n+1}(f)=p(t)$ does not admit a rational solution, then the equation $(E)$ is irreducible.
\end{proposition}

\begin{lemma}
Let $A_1={\scriptsize{\left(\begin{array}{cc} 0 & 1 \\ t & 0 \end{array}\right)}}$ denote the companion matrix of the Airy equation.
Let $A_1^\vee:=-A_1^T$ denote the matrix of the dual system. In a convenient basis, the matrix $\Psi$ of the adjoint action 
$[{\CA}_{diag},\bullet]$ of ${\CA}_{diag}$ on $\mathfrak{h}_{sub}$ is  $\Psi=\sym^{n+1}(A_1)^\vee$.
\end{lemma}

\begin{proof}
We have $$ \sym^{n+1}(A)^\vee = 
	\left(\begin{array}{clcrc} 
		0 & -(n+1)t & \ddots & & 
	\\	-1 & \ddots & -n t & 0 & 
	\\	\ddots & - 2 & \ddots &\ddots &\ddots 	 
	\\	 & 0 & \ddots & \ddots & -  t
	\\	 & & \ddots & -n & 0 
	\end{array}\right).$$
We choose the following basis of $\mathfrak{h}_{sub}$, using again the vector field representation.
Start from the same matrix $F_0:=E_0$ and set $F_{i+1}:=-\frac{1}{i+1} [X,F_i]$. Then, one can check that $[Y,F_i]=-(n+2-i) F_{i-1}$.
So the matrix $\Psi$ of the map $[X+t Y, \bullet]$ on the basis $(F_i)$ is naturally $\sym^{n+1}(A)^\vee$.
\end{proof}

\begin{proofof}{ the Proposition}
Let us go backwards: assume that the equation $y''=x y + y^n P(x,y)$ is reducible. Then we must have $\mathfrak{g}_3=\mathfrak{sl}_2$ (otherwise,  the dimension of the Lie Algebra $\mathfrak{g}_3$ would be $n+5$, thus exceeding the bound of theorem 1).
Reducing to $\mathfrak{sl}_2$ implies that we can find a rational solution to $Y'= \Psi Y + \vec{b}$, where $\vec{b}=(p(t), 0,\ldots, 0)^T$. Transforming the latter to an operator, via the cyclic vector $(0,\ldots,0,1)$ reduces the system to the equation $\sym^{n+1}(\partial_t^2-t)^\vee= p(t)$. But $\partial^2-t$ is self-adjoint, hence the result.
\end{proofof}

\begin{corollary} The equation $(P_2): y'' = xy + 2y^3$ is irreducible. \end{corollary}
\begin{proof} In this case, $n=3$ and $L_4=\partial^5-20 t \partial^3-30 \partial^2+64 t^2 \partial+64 t$. 
A solution of $L_4(y)=12$ would be a polynomial (because $L_4$ has no finite singularity); now the image of a polynomial of degree $N$ by $L_4$ is a polynomial of degree $N+1$ so $12$ cannot be in the image of $L_4$. As equation $L_4(y)=12$  has no rational solution, Proposition \ref{irreducibility-criterion-airy} 
shows that $(P_2)$ is irreducible. \end{proof}

\begin{corollary} Assume that $p(t)$ has a pole of order $k$, $1\leq k \leq n+2$. Then the equation $y''=x y + y^n P(x,y)$ is irreducible.
\end{corollary}
\begin{proof} 
As Airy has no finite singularity, neither does $L_{n+1}$. Thus, if a function $f\in \CC(t)$ has a pole of order $d>0$, then $L(f)$ has this pole of order $d+n+2$.
So if $p(t)$ is in the image of $f$ by $L$ then all its poles have order at least $n+3$.
 \end{proof}

\section{Irreducibility of Painlev\'e III Equations}

The third Painlev\'e \'equation is 
$$
(P_{III}):\; \frac{d^2y}{dx^2} = \frac{1}{y} \left(\frac{dy}{dx}\right)^2 - \frac{1}{x}\frac{dy}{dx} + \frac{1}{x}(\alpha y^2 +\beta) + \gamma y^3
 + \delta \frac{1}{y}$$
 with $(\alpha, \beta, \gamma, \delta) \in \CC^4$.
 For special values $(\alpha, \beta, \gamma, \delta) = (2\mu -1, -2\mu+1, 1,-1)$, $\mu\in \CC$ this equation has a solution: $y=1$. 
 For $\mu = 1/2$, this equation is related to the 2D Ising model in statistical physics, see \cite{MCTW,tracy}. 
 We will show that the latter equation is irreducible (in fact, we prove its irreducibility for $\mu\not\in \ZZ$).
 
 This equation has a time dependent Hamiltonian form (see e.g. \cite{clarkson2010,clarkson2006}). Letting
 $$
 x H(x,y,z) = 2y^2z^2 -(xy^2-2\mu  y -x)z -\mu xy,
 $$
 we may consider the time-dependent Hamiltonian system $\{ \frac{dy}{dx} = \frac{\partial H}{\partial z}, \frac{dz}{dx} = -\frac{\partial H}{\partial y}\}$. Eliminating $z$ between these equations shows that $y$ satisfies $(P_{III})$.
It also means that that solutions of $P_{III}$ parameterize curves $x\mapsto (x, y(x), \frac{xy'(x)+xy(x)^2 -2\mu y(x) -x}{4y(x)^2})$ which are integral curves of the vector field $X = \frac{\partial}{\partial x }+ \frac{\partial H}{\partial z}\frac{\partial}{\partial y} - \frac{\partial H}{\partial y}\frac{\partial}{\partial z}$.

\begin{proposition}
The third Painlev\'e \'equation $(P_{III})$ with parameters $(\alpha, \beta, \gamma, \delta) = (2\mu -1, -2\mu+1, 1,-1)$, where $\mu\not\in \ZZ$,  is irreducible.
\end{proposition}

Before we prove the theorem, we remark that it includes the case $\mu = 1/2$: the third Painlev\'e \'equation $(P_{III})$,
as it appears  in the study of the 2D Ising model in statistical physics in \cite{MCTW,tracy},  is irreducible.

\begin{proof}
This vector field $X$ satisfies the hypothesis of our theorem with the forms $\alpha = dx$, $\gamma = dy\wedge dz + dH \wedge dx$ and the algebraic invariant curve $(\Gamma)$ given by $y= 1, z= -\frac{\mu}{2}$. \\
The first variational equations along $\Gamma$ has matrix
$$ A_1 = \left( \begin {array}{cc} -2-2\,{\frac {\mu}{x}}&  \frac{4}{x}
\\ \noalign{\medskip}-\mu-{\frac {{\mu}^{2}}{x}}&2+2\,{\frac {\mu}{x}}
\end {array} \right).$$
Conjugation by $$Q_1:=\left( \begin {array}{cc} -2\mu \, & 1\\ \noalign{\medskip}-{\mu}^{2}&0
\end {array} \right)$$
puts it in Jordan normal formal at $0$, giving us
$${\tilde A}_1:= Q_1^{-1}.A_1.Q_1 = \left( \begin {array}{cc} 0& \frac1{\mu} + \frac1{x} \\ \noalign{\medskip}4
\,\mu&0\end {array} \right)
=\frac1x \left( \begin {array}{cc} 0&1\\ \noalign{\medskip} 0&0\end {array} \right)
 + \left( \begin {array}{cc} 0& \frac{1}{\mu} \\ \noalign{\medskip}4\,\mu&0\end {array} \right).
$$
We have $Trace({\tilde A}_1)=0$ so $Gal(VE_1) \subset SL(2,\CC)$. 
This first variational equation is equivalent with the differential operator $L_2:=(\frac{d}{dx})^{2}-4-4\,{\frac {\mu}{x}}$. This $L_2$ is 
 reducible for integer $\mu$ (it then has an exponential solution $e^{\pm 2x} P_{\mu}(x)$, where $P_{\mu}$ is a polynomial of degree $|\mu|$)  and it is irreducible otherwise. Moreover, it admits a log in its local solution at $0$, as shown by the Jordan form structure of the local matrix at 0. So, for $\mu\not\in\ZZ$, the Boucher-Weil criterion \cite{BW03} shows that  
$Gal(VE_1)=SL(2,\CC)$ and the first variational equation is in reduced form.
\\
Let $A_2$ be the matrix of the second variational equation. As $A_1$ is in reduced form, we let
$$Q_2:= \left( 
     \begin{array}{c|c}      
          Sym^2(Q_1) &   0_{3\times 2} \\
         \hline
        0_{2\times 3} &{Q_1}
       \end{array}
     \right)
$$
and ${\tilde A}_2:=Q_2^{-1}\cdot A_2\cdot Q_2$. 
We obtain ${\tilde A}_2 = C_{\infty} + \frac1x C_0$, where $C_i$ are constant matrices. Indeed, setting $M_1:=C_0$ and 
$M_2:=C_{\infty} - \frac{1}{\mu} C_0$, we have
$$ M_1= 
	\left( \begin {array}{ccccc} 
		0	&	1	&	0	&	0	&	0\\ 
		\noalign{\medskip}
		0	&	0	&	2	&	0	&	0\\ 
		\noalign{\medskip}
		0	&	0	&	0	&	0	&	0\\ 
		\noalign{\medskip}
-4\,{\mu}^{2}	&	2\,\mu& 	0	& 	0	&	1\\ 
		\noalign{\medskip}
		0	& 4\,{\mu}^{2}&	-2\,\mu&	0	&	0
		\end {array}
	 \right)
 \; \textrm{ and } \; 
 M_2 = \left( \begin {array}{ccccc} 0&0&0&0&0\\ \noalign{\medskip}8\,\mu&0&0
&0&0\\ \noalign{\medskip}0&4\,\mu&0&0&0\\ \noalign{\medskip}0&-1&0&0&0
\\ \noalign{\medskip}-12\,{\mu}^{2}&0&1&4\,\mu&0\end {array} \right).
$$
Now, letting $M_3:=\frac{1}{8\mu} [M_1,M_2]$, a simple calculation shows that $[M_1,M3]=-M_1$ and $[M_2,M_3]=M_2$.
It follows that the Lie algebra $Lie({\tilde A}_2)$ is equal to $SL(2,\CC)$ (in a $5$-dimensional representation). 
It follows that $Gal(VE_2) \subseteq SL(2,\CC)$. However, we know that $SL(2,\CC) \subseteq Gal(VE_2)$ (because $Gal(VE_1)=SL(2,\CC)$)
so $Gal(VE_2) = SL(2,\CC)$ and ${\tilde A}_2$ is in reduced form.
\\

We thus need to go to the third variational equation. Its matrix has the form
$$A_3 = \left(
   \begin{array}{c | c | c}
	\sym^3(A_1) & & \\ \hline 
	B_2^{(3)} & \sym^2(A_1) & \\ \hline 
	B_3^{(3)} & B_2^{(2)} & A_1
  \end{array} \right)
  \quad \textrm{ where } \;
  A_2 = \left(
   \begin{array}{ c | c}
	\sym^2(A_1) & \\ \hline 
	 B_2^{(2)} & A_1
  \end{array} \right).$$
 and $B_2^{(3)}$ comes from $B_2^{(2)}$ so the really new part is the south-west block $B_3^{(3)}$. 
 \\ The situation is strikingly similar to the $(P_{II})$ case from the previous section.
 Let
 $$ 
N_1= \left( \begin{array}{c|c}
  0_{7 \times 7} & 0 _{7 \times 2} \\ \hline
  \begin {array}{rrrr} 0&0&0&0\\ 1&0&0&0
\end {array} 
& 0 _{2 \times 2}
 \end{array}\right),
 \; 
 N_2= \left( \begin{array}{c|c}
  0_{7 \times 7} & 0 _{7 \times 2} \\ \hline
\begin {array}{rrrr} 1&0&0&0\\  0&-1&0&0
\end {array}
& 0 _{2 \times 2}
 \end{array}\right),
 \; 
  N_3= \left( \begin{array}{c|c}
  0_{7 \times 7} & 0 _{7 \times 2} \\ \hline
\begin {array}{rrrr} 0&1&0&0\\   0&0&-1&0
\end {array}
& 0 _{2 \times 2}
 \end{array}\right),
$$ 
$$ 
 N_4= \left( \begin{array}{c|c}
  0_{7 \times 7} & 0 _{7 \times 2} \\ \hline
  \begin {array}{rrrr} 0&0&1&0\\   0&0&0&-1
\end {array}
& 0 _{2 \times 2}
 \end{array}\right),
 \; 
 N_5= \left( \begin{array}{c|c}
  0_{7 \times 7} & 0 _{7 \times 2} \\ \hline
 \begin {array}{rrrr} 0&0&0&1\\    0&0&0&0
\end {array} 
& 0 _{2 \times 2}
 \end{array}\right).
 $$
As in the study of the variational equation, we form a block-diagonal partial reduction matrix $Q_3$ 
with blocks $\sym^3(Q_1),\sym^2(Q_1),Q_1$ and let ${\tilde A}_3=Q_3^{(-1)} \cdot A_3\cdot Q_3$.
Again, we obtain ${\tilde A}_3 = C_{\infty} + \frac1x C_0$, where $C_i$ are constant matrices.
We set $M_1:= C_0$ and $M_2:=\frac{1}{4\mu} C_{\infty} -\frac{1}{\mu} C_0$ and $M_3:=[M_1,M_2]$.
Then, direct inspection shows that $Lie( {\tilde A}_3) = vect_{\CC}( M_1,M_2, M_3, N_1,\ldots, N_5)$ and has dimension 8.
Using the results of the previous section, it follows that we have either $\mathfrak{g}_3 = \mathfrak{sl}(2)$ of dimension 3 or 
$\mathfrak{g}_3 = Lie( {\tilde A}_3)$ of dimension 8.
\\

The adjoint maps $Ad_{M_i}:=[M_i, \bullet]$ acting on $ vect_{\CC}(N_1,\ldots, N_5)$ have respective matrices
$$ \Psi_1 = 
\left( \begin {array}{ccccc} 0&0&0&0&0\\ \noalign{\medskip}1&0&0&0&0
\\ \noalign{\medskip}0&-2&0&0&0\\ \noalign{\medskip}0&0&-3&0&0
\\ \noalign{\medskip}0&0&0&-4&0\end {array} \right) 
\; \textrm{ and } \;
\Psi_2 = 
 \left( \begin {array}{ccccc} 0&4&0&0&0\\ \noalign{\medskip}0&0&-3&0&0
\\ \noalign{\medskip}0&0&0&-2&0\\ \noalign{\medskip}0&0&0&0&-1
\\ \noalign{\medskip}0&0&0&0&0\end {array} \right)
 $$
and $\Psi_3=[\Psi_1,\Psi_2]$ (this follows from the Jacobi identities on Lie brackets).
The matrix of the adjoint action of ${\tilde A}_3$ on $ vect_{k}(N_1,\ldots, N_5)$
is $\Psi:=(\frac{1}{\mu} + \frac{1}{x})\Psi_1 + 4\mu \Psi_2$. 
\\

In order to have $\mathfrak{g}_3 = \mathfrak{sl}(2)$, we would need to find a find a Gauge transformation matrix
$P = Id_{9\times 9} + \sum_{i=1}^5 f_i N_i$ (with $f_i\in \overline{k}$) such that $Lie( P[{\tilde A}_3] ) = vect_{\C}(M_1,M_2,M_3)$.
Let $\vec{b}=(b_1,\ldots,b_5)^T$ be defined by 
${\tilde B}_3^{(3)} = \sum_{i=1}^5 b_i N_i$, namely 
$\vec{b}=(	-32\,{\frac {{\mu}^{4}}{x}}, \; -8\,{\frac {{\mu}^{3}}{x}}, \; 4/3\,{\frac {{\mu}^{2}}{x}} , \;0 , \; 0 )^T$.
Then, letting $\vec{F} = (f_1,\ldots, f_5)^T$, the method developed for $(P_{II})$ in the previous section shows that 
$Lie( P[{\tilde A}_3] ) = vect_{\C}(M_1,M_2,M_3)$ if and only if the $5\times 5$ system $\vec{F}'=\Psi\cdot \vec{F} + \vec{b}$ has an algebraic solution. 
\\

It is easily seen that the latter is impossible. For example, the above system converts to $L(f_1)=g$ where
$g=8192\,{\frac {{\mu}^{4}}{x}}+5120\,{\frac { \left( 4\,\mu+1 \right) {
\mu}^{4}}{{x}^{2}}}+512\,{\frac { \left( 24\,{\mu}^{2}+16\,\mu-7
 \right) {\mu}^{4}}{{x}^{3}}}-256\,{\frac { \left( 31\,\mu+3 \right) {
\mu}^{4}}{{x}^{4}}}+768\,{\frac {{\mu}^{4}}{{x}^{5}}}$
and $L=\sym^4(L_2)$ where $L_2 = (\frac{d}{dx})^2 -4-4\,{\frac {\mu}{x}}$.
When $\mu\not\in\ZZ$ (as assumed here), the differential Galois group of $L_2$ (and hence of $L$) is $SL(2,\CC)$).
So the equation $L(f_1)=g$ has an algebraic solution if and only if it has a rational solution. Let us prove that the latter is impossible.
\\

The exponents of $L$ at zero are positive integers; it follows that, if $f_1$ had a pole of order $n\geq 1$ at zero, $L(f_1)$ would have a pole of order $n+5\geq 6$ at zero. As $g$ only has a pole of order 5 at zero, $f_1$ must be a polynomial. But then $L(f_1)$ would have a pole of order at most $4$ at zero, contradicting the relation $L(f_1)=g$.
\\
Reasoning as in Section 2, 
it follows that the Lie algebra of the Galois group of the third variational equation is 
$\mathfrak{g}_3 = Lie( {\tilde A}_3)vect_{\CC}( M_1,M_2, M_3, N_1,\ldots, N_5)$ and has dimension 8. Our theorem 1 thus implies that 
$(P_{III})$ is irreducible for these values of its parameters.
\end{proof}

\appendix

\centerline{\textbf{\Large -- Appendix --}}

\section{Review on Principal Connections}

The $G$-principal connections are the version of linear differential systems in fundamental form for an algebraic group $G$ that may not be a linear group 
or not be canonically embedded in a $GL_n$. They are a geometric version of Kolchin's strongly normal extensions \cite{kolchin}. The differential systems in vector form appear as a quotient of this fundamental (or principal) form.

\subsection{$G$-Principal Partial Connection}

Consider an algebraic group $G$ and a smooth algebraic manifold $M$.
A \emph{principal $G$-bundle} is a bundle $P \overset{\pi}{\rightarrow}M$ over $M$
such that $G$ acts on $P$ and the map
 $P \times G \to P \underset{M}{\times} P$ given by $(p,g) \mapsto (p,pg)$ is an isomorphism.
 
Let $\Fc$ be an algebraic singular foliation on $M$. 
A \emph{connection along $\Fc$} (or a \emph{partial connection}) on a bundle $P\overset{\pi}{\rightarrow}M$ is a lift of vector field tangent  to $\Fc$ on $P$. 
If $0 \to T(P/M) \to TP \overset{\pi_\ast}{\to} TM \underset{M}{\times} P \to 0$ is the tangential exact sequence then a connection along $\Fc$ is  a splitting above $\Fc$ given by $\nabla: T\Fc\underset{M}{\times} P \to TP$. 
Such a partial connection is called a \emph{rational partial connection} when the splitting is rational.  
\\
We are interested in the case where $\Fc$ is defined by a rational vector field $X$ on $M$. In this situation, it is enough to lift $X$ to $P$ by a rational
vector field $\nabla_X$ such that $\pi_\ast \nabla_X = X$. Then $\nabla$ is defined on vector colinear to $X$ by linearity.

A \emph{$G$-principal connection along $\Fc$} is a $G$-equivariant splitting 
$\nabla: T\Fc\underset{M}{\times} P \to TP$ such that 
$\nabla(X)(pg) =g_\ast \nabla(X)(p)$ where $g_\ast: TP \to TP$ is the map induced by the action of $g$ on $P$.
  
If $G \subset H$ is an inclusion of algebraic groups and $P$ is a $G$-principal bundle then one defines an $H$-principal bundle $HP =(H \times P)/G$ where 
$(h,p)g = (hg,pg)$. A partial $G$-connection $\nabla: T\Fc \underset{M}{\times} P \to TP$ can be composed with the inclusion $H \times TP \subset T(H\times P)$ 
and we obtain a map  $T\Fc\underset{M}{\times}(H\times P) \to T(H \times P)$. 
This map is $G$-equivariant.  By taking quotients, we get
the induced $H$-principal connection along $\Fc$, given by
$H\nabla: T\Fc\underset{M}{\times}(HP) \to T(HP)$. It is the extension of $\nabla$ to $H$.   
In particular, if $G$ is a linear group then the extension of a partial $G$-principal connection to $GL(n,\CC)$ is a usual linear connection in fundamental form with respect to variables tangent to $\Fc$.

\subsection{$G$-Connections and their Galois Groups}

In this paper, a $G$-bundle $E \to M$ will be given by:
 the typical fibre $E_{\ast}$ (an affine variety with an action of $G$), a $G$-principal bundle $P \to M$ and a quotient $E=(P\times E_{\ast})/G$ for the diagonal action of $G$. 
 A principal connection along $\Fc$
on $P$ will induce a connection along $\Fc$ on $E$. Such a connection is called a \emph{partial $G$-connection on $E$}.  

A connection $\nabla$ on a bundle $E \to M$ may be viewed as a $G$-connection for many different groups (and maybe for no group). 
If we know that such a group $G$ exists, we denote by $GE$ the principal bundle and $G\nabla$ the $G$-principal connection. 
The Galois group of the $G$-connection will be a good candidate for such a group
\\

If $\CC(M)^\Fc = \CC$ {\it i.e.} when the foliation has no rational first integrals, then there exists a smallest algebraic group $Gal\nabla \subset G$ such that $\nabla$ is birational to a $Gal\nabla$-connection. This group is well defined up to conjugation in $G$ and is called the \emph{Galois group of $\nabla$}. Its existence is proved following the classical Picard-Vessiot theory in the following way.
A $Gal\nabla$-principal bundle is obtained as a minimal $G\nabla$-invariant algebraic subvariety $Q \subset GE$ dominating $M$ and 
$Gal\nabla$ is the stabilizer of $Q$ in $G$. It is easy to prove that this group is a well defined subgroup of $G$ up to conjugacy.

When a connection $\nabla$ is given, it is not easy to find a group $G$ which would endow $\nabla$ with a structure of  $G$-connection. If such groups exists, we have to prove that our result does not depend on the choice of one of these groups. In the case of linear connections, there is a canonical choice (up to the choice of a point on $M$)

Given a vector bundle $E$, we say that $\nabla$ is a \emph{linear connection} when, for any $X \in \F$, $\nabla(X)$ preserves the module $\E$ of functions on $E$ which are linear on the fibers.
Then there is a canonical way to obtain a principal connection. 
Following Picard and Vessiot, if $E_{m}$ is the fiber of $E$ at $m\in M$ then the tensor product $E \otimes E^{\ast}_m$ of our vector bundle with the dual of the trivial vector bundle $M \times E_{m}$ is endowed with 
\begin{itemize}
\item a connection given by the connection on the first factor,
\item an action of $GL(E_{\ast})$ on the second factor, thus preserving the connection,
\item a canonical point $id \in E_m \otimes E_m^\ast$ in the fiber at $m$.
\end{itemize}   
 Then the space $max(E \otimes E^{\ast}_m)$ of tensors of maximal rank is a $GL(E_{m})$-principal bundle endowed with a principal partial connection. From the action of $GL(E_{m})$ on $E_{m}$, we see that linear connections are $GL$-connections.

The Galois group obtained from a linear connection using this principal bundle and the minimal invariant subvariety $Q$ containing $id$  is called $Gal_m \nabla$. 

\subsection{Principal Bundle and Groupoids}

Given a $G$-principal bundle over $M$, $P \overset{\pi}{\rightarrow}M$, one obtains a groupo\"{\i}d $\G$ 
by taking the quotient $\G:= (P\times P)/G$ of the cartesian product by the diagonal action of $G$ (see \cite{mackenzie} for more details, the main example is described in sections \ref{frame} and \ref{malgrange}). The identity is the subvariety quotient of the diagonal in $P \times P$. 
From a $G$-principal connection $\nabla$ on $P$, one derives a connection $\nabla  \oplus  \nabla$ on the product $P \times P$ defined in a obvious way from the decomposition $T(M \times M) \underset{M \times M}{\times} (P\times P) = TM\underset{M}{\times} P \underset{P}{\oplus} TM\underset{M}{\times} P$. This connection is the so-called \emph{flows matrix equation}.

Let $G\subset H$ be an inclusion of algebraic groups  and $HP \to M$ and $H \nabla$ be an extension of the principal connection to $H$. One gets a groupo\"{\i}d inclusion  $\G \to \H$ such that  $(H\nabla \oplus H\nabla)|_{\G} =\nabla \oplus \nabla $.

\begin{remark}

The following claims are not used in this paper. They may help the reader to understand the links between the various definitions of differential Galois groups appearing in the literature \cite{bertrand,katz,pillay,cartier}.\begin{itemize}
\item The smallest algebraic subvariety of $\G$ which contains the identity and is $\nabla  \oplus  \nabla$-invariant is the Galois groupo\"{\i}d of $\nabla$. 
\item Its restriction above $\{x\}\times M \subset M \times M$ is the Picard-Vessiot extension pointed at $x \in M$.
\item Its restriction over the diagonal $M \subset M\times M$ is a $\D_{M}$-group bundle called the intrinsic Galois group of $\nabla$ in the
sense of Pillay \cite{pillay}. 
\end{itemize}
\end{remark}

\section{Variational Equations}

Various types of variational equations appear in the literature. 
For instance three of them appear in \cite{MRS}. 
 More precisely,  there are various way to obtain a linear system from the variational 
equation seen as an equation on germs of curves. 
In this paper, for the theoretical result, we consider the \emph{frame variational equation} (see below) as the 
principal connection coming from the variational equation. However, for practical calculations, one generally linearizes the variational equation.  

In this appendix, we give the definitions and the proofs needed to compare these different approaches. Some of these results can be found in \cite[propositions 8 to 12]{MRS}.

\subsection{Arc Bundles and the Variational Equation}  

This variational equation does not appear in the article of Morale-Ramis-Sim\'o  \cite{MRS}. 
It is used by several authors (e.g. \cite{BW03}) 
 as a perturbative variational equation.

The set of all parameterized curves on $M$ is denoted by $ CM = \{ c: \widehat{(\CC,0)} \to M \}$.
 It has a natural structure of pro-algebraic
variety. Let $\CC[M]$ be the coordinate ring of $M$ and $\CC[\delta]$ be the $\CC$-vector space of linear ordinary differential 
operators with constant coefficients. The coordinate ring of $CM$ is $Sym(\CC[M]\otimes \CC[\delta]) /L$ where 
\begin{itemize}
\item the tensor product is a tensor product of $\CC$-vector spaces,
\item $Sym( V )$ is the $\CC$-algebra generated by the vector space $V$, 
\item $Sym(\CC[M]\otimes \CC[\delta])$ has a structure of $\delta$-differential algebra {\it via}
the right composition of differential operators,
\item the Leibniz ideal $L$ is the $\delta$-ideal generated by $fg\otimes 1 - (f\otimes1)(g\otimes1)$ for all $(f,g) \in \CC[M]^2$.
 \end{itemize}

Local coordinates $(x_1, \ldots, x_n)$ on $M$ induce local coordinates on $CM$ {\it via} the Taylor expansion of curves $c$ at $0$:
$$
c(\epsilon) = \left( \sum c_1^{(k)} \frac{\epsilon^k}{k!}, \ldots,\sum c_n^{(k)} \frac{\epsilon^k}{k!}   \right).
$$ 
Let  $x_i^{(k)}: CM \to \CC$ be defined by $x_i^{(k)}(c) = c_i^{(k)}$. This function is the element $x_i\otimes \delta^k$ in 
$\CC[CM]$ and we have the following facts.

\begin{enumerate}
\item $\CC[CM]$ is the $\delta$-algebra generated by $\CC[M]$. The action of $\delta: \CC[CM] \to \CC[CM]$ can be written in
local coordinates and gives the total derivative operator $\sum_{i,k} x_i^{(k+1)} \frac{\partial}{\partial x_i^{(k)}}$.
\item Morphisms (resp. derivations) of $\CC[M]$ act on $\CC[CM]$ as morphisms (resp. derivations) {\it via} the first factor (it can be easily checked that the Leibniz ideal is preserved)
\item The vector space $\CC[\delta]$ is filtered by  the spaces $\CC[\delta]^{\leq k}$ of operators of order less than $k$. This gives a
filtration of $\CC[CM]$ by  $\CC$-algebras of finite type. 
\item These algebras are coordinate ring of the space of $k$-jet of parameterized curves $C_kM =\{j_k c \ |\ c \in CM\}$.
\item The action of $\delta$ has degree $+1$ with respect to the filtration.
\item  Prolongations of morphisms and derivations of $\CC[M]$ on $\CC[CM]$ have degree $0$.
\end{enumerate}

Set theoretically, the prolongations are obtained in the following way.
Any holomorphic map $\varphi: U \to V$ between open subset of $M$ can be prolonged on open sets $C_kU$ of $C_kM$ of curves through 
points in $U$ by $C_k\varphi: C_kU \to C_kV; j_kc \mapsto j_k(\varphi \circ c)$. One easily checks that 
$C_k(\varphi_1 \circ \varphi_2) = C_k\varphi_1 \circ C_k\varphi_2$.
This equality can be used to prolong holomorphic vector fields defined on open subsets $U \subset M$ by the infinitesimal generator of the local 1-parameter
group obtained by prolongation of the flow of $X$ {\it i.e.} $C_k(\exp(tX)) = \exp(tC_kX)$. 

When $X$ is a rational vector field on $M$, its prolongation $C_kX$ on $C_kM$ is also rational. 
Let $X = \sum a_i(x)\frac{\partial}{\partial x_i}$ be a vector field on $M$ in local coordinates. One gets 
$C_kX = \sum_{i,\ell \leq k} \delta^\ell(a_i)\frac{\partial}{\partial x_i^\ell}.$ 

If $\Fc$ is the foliation by integral curves of $X$ on $M$ then $C_{k}X$ defines a rational connection along $\Fc$ on $C_kM$: 
for  a  vector $V$ tangent to $\Fc$ at $x\in M$ with $X(x)\not = 0$ or $\infty$, one defines $\nabla_V(j_kc) = \frac{V}{X(x)} R_kX (j_kc)$. It is the 
\emph{order $k$ variational connection/equation} of X.

Usually, the variational equation is studied along a given integral curve of $X$: if $\Cc$ is an invariant curve and if 
$C_kM_\Cc$ is the subspace of $C_kM$ of curves through points in $\Cc$ the vector field $C_kX$ preserves $C_kM_\Cc$. 
Its restriction to $C_kM_\Cc$ is called \emph{the order $k$ variational connection/equation along $\Cc$}.

\subsection{Frame Bundles and the Frame Variational Equation}
\label{frame}
This variational equation is the one used in the theoretical part of \cite[\S 3.4]{MRS} as well as in \cite{casale-MR}.

The set of all formal frames on $M$ is denoted by $ RM = \{ r: \widehat{(\CC^n,0)} \to M | \det(Jac(r))\not = 0 \}$.
Like the arc spaces, it has a natural structure of pro-algebraic
variety. Let $\CC[\partial_1, \ldots, \partial_n]$ be the $\CC$-vector space of linear partial differential 
operators with constant coefficients. The coordinate ring of $RM$ is $\left(Sym(\CC[M]\otimes \CC[\partial_1,\ldots,\partial_n]) /L\right) [1/Jac]$ where 
\begin{itemize}
  \item the tensor product is a tensor product of $\CC$-vector spaces;
  \item $Sym( V )$ is the $\CC$-algebra generated by the vector space $V$;
  \item $Sym(\CC[M]\otimes \CC[\partial_1,\ldots, \partial_n])$ has a structure of $\CC[\partial_1,\ldots,\partial_n]$-differntial algebra {\it via} 
          the right composition of differential operators;
  \item the Leibniz ideal $L$ is the $\CC[\partial_1,\ldots, \partial_n]$-ideal generated by $fg\otimes 1 - (f\otimes1)(g\otimes1)$ for all $(f,g) \in \CC[M]^2$;
  \item the quotient is then localized by $Jac$ the sheaf of ideals (not differential !) generated by $\det\left(  [x_i \otimes \partial_j]_{i,j}\right)$
          for a transcendental basis $(x_1, \ldots, x_n)$ of $\CC(M)$ on a Zariski open subset of $M$ where such a basis is defined. 
 \end{itemize}
Local coordinates $(x_1, \ldots, x_n)$ on $M$ induce local coordinates on $RM$ {\it via} the Taylor expansion of maps $r$ at $0$:
$$
r(\epsilon_1\ldots, \epsilon_n) = \left( \sum r_1^{\alpha} \frac{\epsilon^\alpha}{\alpha!}, \ldots,\sum r_n^{\alpha} \frac{\epsilon^\alpha}{\alpha!}   \right).
$$ 
Let $x_i^{\alpha}: RM \to \CC$ be defined by $x_i^{\alpha}(r) = r_i^{\alpha}$. 
This function is the element $x_i\otimes \partial^\alpha$ in 
$\CC[RM]$.

\begin{enumerate}
\item The action of $\partial_j: \CC[RM] \to \CC[RM]$ can be written in
local coordinates and gives the total derivative operator $\sum_{i,\alpha} x_i^{\alpha + 1_j} \frac{\partial}{\partial x_i^{\alpha}}$.
\item We leave to the reader the translation of the properties from section B.1 in this multivariate situation.
\end{enumerate}
All the remarks we have made about arc spaces extend {\it mutatis mutantis} to the frame bundle. There is one important difference: 
$RM$ is a principal bundle over $M$. Let us describe this structure here.

The pro-algebraic group $$\Gamma = \{ \gamma: \widehat{(\CC^n,0)} \overset{\sim}{\rightarrow} \widehat{(\CC^n,0)}; \text{ formal invertible}\}$$ is the projective limit of groups $$\Gamma_k = \{ j_k \gamma | \gamma:(\CC^n,0) \overset{\sim}{\rightarrow} (\CC^n,0); \text{ holomorphic invertible}\}.$$ 
It acts on $RM$ and the map $ RM \times \Gamma  \to RM \underset{M}{\times} RM$ sending $(r,\gamma)$ to $(r,r\circ \gamma)$ is an isomorphism.
The action of $\gamma \in \Gamma$ on $RM$ is denoted by $S\gamma: RM \to RM$ as it acts as a change of
source coordinates of frames. At the coordinate ring level, this action is given by the action of formal change of coordinates on 
$\CC[\partial_1,\ldots,\partial_n]$ followed by the evaluation at $0$ in order to get an operator with constant coefficients.  
   This action has degree $0$ with respect to the filtration induced by the order of differential operators. 
   For any $k$, this means that the bundle of order $k$ frames $R_kM$ is a principal bundle over $M$ for the group $\Gamma_k$.

When $X$ is a rational vector field on $M$, its prolongation $R_kX$ on $R_kM$ is also rational. 
Let $X = \sum a_i(x)\frac{\partial}{\partial x_i}$ be a vector field on $M$ in local coordinates. One gets 
$R_kX = \sum_{i, |\alpha|\leq k} \partial^\alpha(a_i)\frac{\partial}{\partial x_i^\alpha}.$ 

If $\Fc$ is the foliation by integral curves of $X$ on $M$ then $R_{k}X$ defines a rational connection along $\Fc$ on $R_kM$. Moreover
the prolongation is defined by an action of the first factor on a tensor product whereas $\Gamma_k$ acts on the other factor.
These two actions commute, meaning that $R_kX$ is a $\Gamma_k$-principal connection along $\Fc$. It is the 
\emph{order $k$ frame variational connection/equation} of X.

As for variational equations, one can restrict this connection above an integral curve $\Cc$ of $X$: one gets
the \emph{the order $k$ frame variational connection/equation along $\Cc$}. This connection is a principal connection of a bundle
on $\Cc$.
After choosing a point $m\in \Cc$ where $X$ is defined, we obtains a Galois group $Gal_m(R_kX|_\Cc) \subset \Gamma_k$.

From a frame $r: \widehat{(\CC^n,0)} \to M$, one derives many parameterized curves: $CM$ is a $\Gamma$-bundle. More precisely: 
if $V_k$ denotes the vector space of $k$-jets of maps $\widehat{(\CC,0)} \to \widehat{(\CC^n,0)}$ then $C_{k}M = (R_{k}M \times V_{k})/\Gamma$.
The order $k$ variational connection is a $\Gamma_{k}$-connection.

\subsection{The Linearized Variational Equations}

The variational equations are usually given in the linearized form described in Section 1.2. In Morales-Ramis-Sim\'o , another linear variational equation
is introduced, using a faithfull linear representation of $\Gamma$. Let us recall these constructions and their relations with the frame variational
equations above.

\subsubsection{The Morales-Ramis-Sim\'o Linearization}
The theoretically easier linearization of variational equation is done through linearization of frame variational equations.
This is the approach followed by Morales-Ramis-Sim\'o. It is based on the fact that  
 $\Gamma_k$ is a linear group. Let $V_k$ be the set of order $k$ jets of map form
$(\CC^n,0) \to (\CC^n,0)$ without invertibility condition. Using coordinates on $(\CC^n,0)$ one can check that $V_k$ is a vector space 
(the addition depend on the choice of coordinates) and, using Faa di Bruno formulas, one can check that $(j_ks, j_k\gamma) \mapsto j_k(s\circ \gamma) $ defines a 
faithfull representation of $\Gamma_k$ on $V_k$. Then, from this inclusion $\Gamma_k \subset GL(V_k)$, one gets an extension of the principal variational equation
to the \emph{jet-linearized principal variational connection}.

\subsubsection{The Geometric Explanation of the Linearization}
The second linearization (see Section 1.) is done in the following way.
The coordinate ring of the arc space has a natural degree from the filtration. It is defined on generators of the algebra by 
$d^\circ(f\otimes\delta^i) = i$. The Leibniz rule implies that this the Leibniz ideal $L$ (see B.2) is generated by homogeneous elements and then the degree is well defined on $\CC[CM]$.

This degree gives a decomposition  $\CC[CM] = \oplus_{k} \E_{k}$ in subspaces of homogeneous functions of degree $k$ called jet differentials of degree $k$ \cite{GG}. It is a straightforward to verify the following properties:
\begin{itemize}
  \item $\E_{k}$ is a locally free $\CC[M]$-module of finite rank;
  \item if $\varphi: U \to V$ is a biholomorphism on open sets of $M$ then $C\varphi$, sending $\O_{V} \underset{\CC[M]}{\otimes} \CC[CM]$ to $\O_{U} \underset{\CC[M]}{\otimes} \CC[CM]$ preserves $\O_{M}\otimes \E_{k} $;
  \item if $X$ is an holomorphic vector field on a open set $U$ of $M$ then $\E_{k}$ is $CX$-invariant.
\end{itemize}
Now $X$ is a rational vector field on $M$ and  $E_{k}$ is the dual vector bundle of $\E_{k}$. 
From these properties, we find that $E_{k}$, endowed with the action of $X$ through $CX$, is a 
linear $\Fc$- connection. It is the \emph{degree-linearized order $k$ variational equation}.

The choice of an invariant curve $\Cc$ and a point $c \in \Cc$, such that $X_c$ is defined and not zero, will give the Galois group of the 
degree-linearized variational equation along $\Cc$ at $c$ denoted by $Gal_{c}(LV_{k}\Cc)$ (even though it depends on $X$).
 
The right composition with $\partial$ gives an inclusion $\E_{k} \to \E_{k+1}$ and thus a projection $Gal_{c}(LV_{k+1}\Cc)$ onto $Gal_{c}(LV_{k}\Cc)$.
The inductive limit of differential systems is denoted by $LV_\Cc$, it is the \emph{degree-linearized variational equation}. 
The projective limite of groups is denoted by $Gal_{c}(LV\Cc)$.

\begin{proposition}
\label{comparaison1}
  The Galois group of the degree-linearized variational equation is isomorphic to the Galois group of the frame variational equation.
\end{proposition}

Theproof will be given in the proofs section.

\subsection{The Covariational Equations}

This variational equation is the one used in \cite[p 861]{MRS} to linearize the variational equation.

The set of all formal functions on $M$ is denoted by $ FM = \{ f: \widehat{(M,m)} \to \widehat{(\CC,0)}\}$. Its structural ring is $\CC[FM] = Sym(\D^{\geq 1}_{M})$, the $\CC[M]$-algebra generated by the module of differential operators on $M$ generated as an operators algebra by derivations. From this definition, $FM$ is the vector bundle over $M$ dual of $\D^{\geq 1}_{M}$. It is a projective limit of $F_{k}M$ the bundle of $k$-jets of functions (the dual of operator of order less than $k$).

A vector field $X$ on $M$ acts on $\D^{\geq 1}_{M}$ by the commutator $P \mapsto [X,P]$ and this action preserves the order. This gives a linear $\Fc$-connection on each $F_{k}M$. This is the linearized variational equation of \cite{MRS}. 
In this paper, it will be called the \emph{covariational equation}. The following comparison result is proved in Appendix C.1 below:

\begin{proposition}
\label{comparaison2}
  The covariational equation of order $k$ and the variational equation of order $k$ have the same Galois group.
\end{proposition}
%


\subsection{Normal Variational and Normal Covariational Equations.}

When the vector field $X$ preserves a foliation $\Gc$ on $M$ then its prolongation $C_{k}X$ on the space $C_{k}M$ of $k$-jets of parameterized curves on $M$ preserves the subspace $C_{k}\Gc$ of curves contained in leaves of $\Gc$.
The subspace $C_{k}\Gc$ is an algebraic subvariety of $C_{k}M$ and the restriction of $C_{k}X$ on $C_{k}\Gc$ is the order $k$ variational equation tangent to $\Gc$. 
When $\Gc$ is generically transversal to the trajectories of $X$, this equation is called the \emph{order $k$ normal variational equation}. We don't know how this equation depends on the choice of such a foliation $\Gc$. However, in our situation of a vector field given by a differential equation, there is a canonical transversal foliation given by the levels of the independent variable.

Let $B$ be the curve with local coordinate $x$ and $\pi: M \dasharrow B$ be the phase space of a differential equation with independent variable $x$. The foliation $\Gc$ is given by the level subsets of $\pi$. Using local coordinate $x_{1}, \ldots, x_{n}$ on $M$ such that $x_{1}=x$, the subvariety $C_{k}\Gc \subset C_{k}M$ is described by the equations $x_{1}^\ell =0, 1\leq \ell\leq k$. The variational equation in local coordinate is the system (\ref{variational system}) page \pageref{variational system}. 
By setting $x_{1}^\ell =0, 1\leq \ell\leq k$ into this system, one gets the differential system for the normal variational equation. 

The linearization of the normal variational equation is done by the linearization of the variational equation. 
Let $I \subset \CC[C_{k}M]$ be the ideal defining the subvariety $C_{k}\Gc$. 
Then $\Ec_{k}\cap I \subset \Ec_{k} \subset \CC[C_{k}M] $ are finite rank linear spaces invariant under the action of $C_{k}X$. The induced action on the quotient $\Ec_{k} / (\Ec_{k}\cap I)$ is the \emph{linearized order $k$ normal variational equation}.\\

The normal covariational equation is more intrinsic. 
Let $F^X_{k} M \subset F_{k}M$ be the space of $k$-jets of first integrals $f$ of X, $f:\widehat{(M,x)} \to \widehat{(\CC,0)}$ such that $X.f =0$. It is a linear subspace defined by its annihilator $\D_{M}^{\geq 1}\cdot X \subset \D_{M}^{\geq 1}$. 
The commutator $P \to [X,P]$ preserves $\D_{M}^{\geq 1}\cdot X $. So,  it defines a linear connection on $F^X_{k} M$: 
this is the \emph{normal covariational equation}.

\section{The Proofs}

We recall the definitions and results of Casale in \cite{casale-irred} using the frame bundle $RM$ of $M$. It has a central place in the theory. 
In this section, it is used to present the Malgrange pseudogroup
and in the previous one it was used to have the variational equation in fundamental form. 

As it is a principal bundle, it has an associated groupo\"{\i}d: $Aut (M) = (RM \times RM)/\Gamma$. 
The $\Gamma$-orbit of a couple of frames $(r,s)$ is the set of all $(r\circ \gamma, s \circ \gamma)$ for $\gamma \in \Gamma$. 
It is characterized by the formal map $r \circ s^{-1}: \widehat{(M,s(0))} \to \widehat{(M,r(0))}$. The quotient $Aut (M)$ is the space of formal selfmaps on $M$ with its natural structure of groupo\"{\i}d. For an $m\in M$,  
we define $Aut (M)_{m,M}$ to be the space of maps with source at $m$ and target anywhere on $M$. The choice of a frame $r: \widehat{(\CC^n,0)} \to \widehat{(M,m)}$ gives a isomorphism between $Aut(M)_{m,M}$ and $RM$.

\subsection{Proofs of the Comparison Propositions}

\begin{proofof}{proposition \ref{comparaison1}}
We will first compare these variational equations for a fixed order, then study their projective limits.

In order to compare all the variational equations, we will need to look more carefully at the frame bundle. The proof is then just another way to write the properties of $\Ec_{k}$. Its second property says that we have a group inclusion $Aut_{k}(M)_{c,c} \to GL(E_{k}(c))$ and a compatible inclusion of principal bundles
  $$
    Aut_k(M)_{c,\Cc} \to E_{k}(\Cc)\otimes (E_{k}(c))^*.
  $$
This inclusion is compatible with the action of the vector field $X$. This means that the fundamental form of the order $k$ degree-linearized variational equation ({\it i.e.} $C_{k}X$ action on  $max\left(E_{k}(\Cc)\otimes (E_{k}(c))^*\right)$) is an extension of the frame variational equation. Thus, their Galois groups are the same.

The comparison of limit groups is not direct because the family $(GL(E_{k}(c))_{k}$ is not a projective system. 
The module $\E_{k}$ is filtered by 
$$
  \E_{0}\circ \delta^k \subset \E_{1} \circ \delta^{k-1} \subset \ldots \subset \E_{k}.
$$
Let $T_{k} \subset GL(E_{k}(c))$ denote the subgroup preserving this filtration. Now: 
\begin{itemize}
  \item there is a natural projection $T_{k} \to T_{k-1}$,
  \item the Galois group of the order $k$ degree-linearized variational equation is a subgroup of $T_{k}$,
  \item and the inclusion $Aut_{k}(M)_{c,c} \to T_{k}$ is compatible with the projections.
\end{itemize}
This proves the proposition.
\end{proofof}

\begin{proofof}{proposition \ref{comparaison2}}
There is a direct way to see that the variational equation and the covariational equation will have the same Galois group. Instead of using the Picard-Vessiot principal bundle for the covariational equation, one can build a better principal bundle. 
Consider the bundle of coframes  $R^{-1}M = \{ f: \widehat{(M,m)} \to \widehat{(\CC^n,0)} \text{ invertible}\}$ whose coordinate ring is $Sym(\D^{\geq 1}_{M}\otimes \CC^n)[1/jac]$. 
This is a $\Gamma$-principal bundle. The action of $X$ by the commutator defines a $\Gamma$-principal connection. This connection is called the \emph{coframe variational equation}. 
The map $R \to R^{-1}$ sending a frame $r$ on its inverse $r^{-1}$ is an isomorphism of principal bundles (up to changing the side of the group action) conjugating the frame and the coframe variational equations.

Now let $F_{k}$ be the vector space of $k$-jets of formal maps $\widehat{(\CC^n,0)} \to \widehat{(\CC,0)}$ and $C_{k}$ be the vector space of $k$-jets of formal maps   $\widehat{(\CC,0)} \to \widehat{(\CC^n,0)}$. 
One has $F_{k}M = (R^{-1}_{k}M \times F_{k})/\Gamma_{k}$ and $C_{k}M = (R_{k}M \times C_{k})/\Gamma_{k}$ . 
Moreover these two isomorphisms are compatible with the various variational equations. So, the Galois group of the covariational equation equals the one of the variational equation. 
\end{proofof}

\subsection{The Malgrange Pseudogroup} 
\label{malgrange}
The \emph{Malgrange pseudogroup of a vector field $X$ on $M$} is a subgroupo\"{\i}d of $Aut(M)$. 
We choose here to call it a pseudogroup as its elements are formal diffeomorphisms between formal neighbourhood of points of $M$ satisfying the definition of a pseudogroup, see \cite{casale-MR}. 

It is defined by means of differential invariants of $X$ {\it i.e.} rational functions $H \in \CC(RM)$ 
such that $RX \cdot H = 0$. Let $Inv(X) \subset \CC(RM)$ be the subfield of differential invariants of $X$. 
Let $W$ be a model for $Inv$ and $\pi: RM \dasharrow W$ be the dominant map from the inclusion $Inv \subset \CC(RM)$. Let $Mal(X)$ be $(RM \underset{W}{\times} RM)/\Gamma \subset AutM$. 
To define properly this fiber product, one needs to restrict $\pi: (RM)^o \to W$ on its domain of definition. Then,  $RM \underset{W}{\times} RM$ is defined to be the Zariski closure of  $(RM)^o \underset{W}{\times}(RM)^o$ in $RM \times RM$. By construction, any Taylor expansion of the flow of $X$ belongs to $Mal(X)$. 

Malgrange shows in  \cite{malgrange} that there exists a Zariski open subset $M^{o}$ of $M$ such that the restriction on $Mal(X)$ to $Aut(M^o)$ is a subgroupo\"{\i}d. This result was extended by Casale in \cite{casale-MR} and allows to view 
the Malgrange pseudogroup as a set-theoretical subgroupo\"{\i}d of $Aut(M)$. \\

From the Cartan classification of pseudogroups, \cite{cartan}, one gets the following theorem for rank two differential system, see the appendix of \cite{casale-P1} for a proof. 

\begin{theorem}[\cite{cartan,casale-P1}] \label{theoreme-casale-MR}
Let $M$ be a smooth irreducible algebraic $3$-fold over $\CC$ and $X$ be a rational vector field on $M$ such that there exist a closed rational $1$-form $\alpha$ with 
$\alpha(X) = 1$ and a closed rational $2$-form $\gamma$ with $\iota_X\gamma = 0$.
One of the following statement holds.
 \begin{itemize}
       \item On  a covering $\widetilde{M} \overset{\pi}{\rightarrow} M$ of a Zariski open subset of $M$, there exists a rational 1-form $\omega$ 
such that $\omega \wedge d\omega =0$ and $\omega(\pi^\ast X) = 0$.   
       Then $$Mal(X) \subset \{\varphi \ | \ \varphi^\ast \alpha = \alpha, (\widetilde{\varphi}^\ast\omega) \wedge \omega =0  \}$$ where $\widetilde{\varphi}$ stand for any lift of $\varphi$ to $\widetilde{M}$. The vector field is said to be 
\emph{transversally imprimitive}.

       \item There exists a vector of rational 1-forms $\Theta = \left( \begin{array}{c} \theta_1 \\ \theta_2 \end{array}\right)$ such that $\theta_1 \wedge \theta_2= \gamma$ and a trace free matrix of 1-forms $\Omega$ such that
       $$
       d\Theta = \Omega \wedge \Theta \ \ \text{ and } \ \  d\Omega = -\Omega \wedge \Omega.
       $$
        One has
       $Mal(X) \subset \{\varphi \ | \ \varphi^\ast \alpha = \alpha, \varphi^\ast\Theta = D\Theta \text{ and } dD = [D, \Omega]  \}$. The 
vector field is said to be \emph{transversally affine}.

       \item $Mal(X) =\{\varphi \ | \ \varphi^\ast \alpha = \alpha, \varphi^\ast\gamma = \gamma \}$.
    \end{itemize}

    \end{theorem}

In order to compute dimensions of Malgrange pseudogroups and of Galois groups of variational equations, it will be easier to work with the Lie algebra of the Malgrange pseudogroup. Roughly speaking, $\mathfrak{mal}(X)$ is the sheaf of Lie algebra of vector fields whose flows belongs to $Mal(X)$. The reader is invited to read \cite{malgrange} for a formal definition.

\subsection{Proof of our Main Theorem \ref{main-theorem}.}

\begin{proof}
The assumptions made on $X$, $\alpha$ and $\gamma$ ensure that the first variational equation is reducible to a block-diagonal matrix with a first block in $SL_{2}(\CC)$ and a second equal to $[1]$.

Using theorems above, the proof of this theorem is reduced to the following  two lemmas.

\begin{lemma}
If there exists  a finite map $\pi: \widetilde{M} \to M$ and an integrable $1$-form on $\widetilde{M}$ 
vanishing on $\pi^\ast X$  then the Lie algebra of the Galois group of the first variational equation along any solution is solvable.
\end{lemma}

\begin{proof}
This lemma is proved in the spirit of Ziglin and Morales-Ramis like Audin does in \cite{audin}. 

First we have to descend from the covering to $M$. Remark that $\omega$ is a $1$-form on $M$ whose coefficients are algebraic functions. 
Letting $\omega_{i}$, $i=1,\ldots, \ell$ be the conjugates of $\omega$, the product $\overline{\omega} = \prod \omega_{i} \in Sym^\ell \Omega^1(M)$ is a well defined rational symmetric form on $M$. 
For any holomorphic function $f$ on open subsets of $M$, $f\overline{\omega}$ satisfies all hypotheses so that one can assume that the rational symmetric form $\overline{\omega}$ is holomorphic at the generic point of $\Cc$. 

At generic $p\in \Cc$, the section $\overline{\omega}: M \dasharrow Sym^\ell T^\ast M$ vanishes at the order $k$ thus one can write $\overline{\omega} = \overline{\omega}_k(p) + \ldots$ where $\overline{\omega}_k(p)$ 
is lowest order homogeneous part of $\overline{\omega}$. 
It is a well defined symmetric form $\ell$ on the tangent space $T_pM$, i.e $\overline{\omega}_{k}(p): T_{p}M \to Sym^{\ell}T^{\ast}(T_{p}M)$; note that 
$\overline{\omega}_{k}(p) (\lambda v) (\mu w) = \lambda^k \mu^\ell \overline{\omega}_{k}(p) (v)$ for any $v \in TpM$ and any $w \in T_{v}(T_{p}M)$. We may say that $\overline{\omega}$ is a symmetric $\ell$-form on the space $T_{p}M$, homogeneous of degree $k$.

\begin{sublemma}
The vanishing order of $\overline{\omega}$ is constant on a Zariski open subset of $\Cc$.
\end{sublemma}

\begin{proof}
Let $p\in \Cc$ be such that $X(p) \not = 0$. One can choose rectifying coordinates $x_{1}, \ldots, x_{n}$ such that $x(p) = 0$, $X = \frac{\partial}{\partial x_{1}}$ and $\displaystyle \overline{\omega} = \sum_{\alpha \in \times\NN^{n-1} \atop |\alpha|=\ell}  w_{\alpha}(x) \prod_{i=2}^n dx_{i}^{\alpha_{i}}$. 
Since $\Lc_{X} \overline{\omega} \wedge \overline{\omega} = 0$, one has that, for any $c\in \CC$ small enough, $w_{\alpha}(x_{1}+c, \ldots, x_{n}) = f_{c}(x)w_{\alpha}(x_{1}, \ldots, x_{n})$ where $f_{c}$ is a holomorphic function depending on $c$ not on $\alpha$. Now $f_{0}=1$ thus $f_{c}(0)\not = 0$ for $c$ small enough and the vanishing order of $\omega$ at $0$ equals the one at $(c,0,\ldots,0)$.
\end{proof}

This sublemma enables us to define a rational section $\overline{\omega}_{k}: T_{\Cc}M \dasharrow Sym^{\ell}V^{\ast}(T_{\Cc}M)$ where $V^{\ast}(T_{\Cc}M) = T^{\ast}(T_{\Cc}M) /T^\ast \Cc$.

Remember that from $X$, we get a vector field $C_{1}X$ on $T_{\Cc}M$ called the first order variational equation along $\Cc$.

\begin{sublemma}
$\Lc_{C_{1}X}\overline{\omega}_{k} \wedge \overline{\omega}_{k} =0$
\end{sublemma}

\begin{proof}
Here again, we will prove it in local analytic coordinates. Let $p\in \Cc$ be such that $X(p) \not = 0$. One can choose rectifying 
coordinates $x_{1}, \ldots, x_{n}$ such that $x(p) = 0$, $X = \frac{\partial}{\partial x_{1}}$ and 
$\displaystyle \overline{\omega} = \sum_{\alpha \in \times\NN^{n-1} \atop |\alpha|=\ell}  w_{\alpha}(x) \prod_{i=2}^n dx_{i}^{\alpha_{i}}$. 
For any $c\in \CC$ small enough,  $w_{\alpha}(x_{1}+c, \ldots, x_{n}) = f_{c}(x)w_{\alpha}(x_{1}, \ldots, x_{n})$ 
so the zero set of $\overline{\omega}$ in $T_{\Cc}M$ is a subvariety  invariant under translations colinear to $X$. 
One can get local equations for this zero set in the form 
$\eta = \sum_{\alpha \in \times\NN^{n-1} \atop |\alpha|=\ell}  n_{\alpha}(x_{2},\ldots,x_{n}) \prod_{i=2}^n dx_{i}^{\alpha_{i}}$ 
and there exists a holomorphic $h$ such that $\overline{\omega} = h \eta$. Now, by taking the lowest order homogeneous parts, 
one gets $\overline{\omega}_{k} = h_{k_{1}} \eta_{k_{2}}$. Since $\eta$ is $x_{1}$-independent so is $\eta_{k_{2}}$. 
In local coordinate induced on $T_{\Cc}M$, $C_{1}X = \frac{\partial}{\partial x_{1}}$ then $\Lc_{C_{1}X} \eta_{k_{2}} = 0$ and a direct computation proves that
$\Lc_{C_{1}X}\overline{\omega}_{k} \wedge \overline{\omega}_{k} =0$.
\end{proof}

\begin{sublemma}
$Gal(C_{1}X)$ is virtually solvable.
\end{sublemma}

\begin{proof}
The rational form $\overline{\omega}_{k}$ defines  in each fiber of $T_\Cc M$ a homogeneous $\ell$-web. This fiberwise rational web is $C_{1}X$-invariant. This implies that the action of the Galois group on a fiber $T_p M$ must preserve this web. In other words, the Galois group at $p$ preserves the set of symmetric forms on $T_{p}M$ which are rational multiples of $\overline{\omega}_{k}(p)$. 

The form $\eta$ given in the previous sublemma shows that the web is a pull-back of a web defined on the normal bundle of $\Cc$ in $M$. The group $Gal(C_{1}X)$ is included in a block diagonal group with a block $(1)$ and a $2 \times 2$ block given by a subgroup of $SL(2,\CC)$.  As $SL(2,\CC)$ does not preserve a web on $\CC^2$, the $2 \times 2$ block is a proper subgroup of $SL(2,\CC)$. This proves the sublemma.

\end{proof} \end{proof}

\begin{lemma}
If $X$ is tranversally affine then the Galois group ot the formal variational equation along any solution has dimension smaller than $5$.
\end{lemma}

\begin{proof}
We will see that this lemma is a consequence of theorems and lemmas from \cite{casale-MR}. 
From theorem 2.4  of \cite{casale-MR}, the Galois group of the formal variational equation along $\Cc$ is a subgroup of $Mal(X)_{p} = \{\varphi: (M,p) \to (M,p) | \varphi \in Mal(X)\}$ for a generic $p\in \Cc$.  Its Lie algebra is included in $$\mathfrak{mal}(X)^0_{p} = \{Y \text{ vector field on } (M,p) | Y(p)=0 , Y \in \mathfrak{mal}(X)\}.$$

From Lemma 3.8 of \cite{casale-MR}, the dimension of $\mathfrak{mal}(X)_{p} = \{Y \text{ vector field on } (M,p) | Y \in \mathfrak{mal}(X)\}$ for $p\in \Cc$ is smaller that the dimension of the same Lie algebra for generic $p \in M$.

Assume $X$ is transversally affine and choose a point $p\in M$ such that the 1-form $\alpha$ and the forms $\Theta^{0}$,  
$\Theta^{1}$ from the definition are holomorphic and $\alpha \wedge \theta^0_{1}\wedge \theta^0_{2} \not =0$. 
Then one can choose local analytic coordinates such that $\alpha = dx_{1}$, and $\begin{bmatrix}d x_{2} \\ dx_{3} \end{bmatrix} = F \Theta^0$ with $dF + F\Theta^1 = 0$. 
In these coordinates, $\varphi \in Mal $ satisfies $\varphi^\ast \alpha = \alpha$, $\varphi^\ast\Theta^0 = D\Theta^0$  and $dD = [D, \Theta^1]$ if and only if 
$\varphi (x_{1}, x_{2}, x_{3}) = (x_{1}+ c_{0}, c_{1}x_{2}+c_{2}x_{3}+c_{3}, c_{4}x_{1}+ c_{5}x_{2} +c_{6})$
with $c\in \CC^7$ such that $\det \begin{bmatrix} c_{1} & c_{2} \\ c_{4} & c_{5} \end{bmatrix}=1$.
The infinitesimal version of these calculations shows that the dimension of $\mathfrak{mal}(X)_{p}$ is smaller that $6$. The Lie algebra $\mathfrak{mal}(X)^0_{p}$ is strictly smaller that  $\mathfrak{mal}(X)_{p}$ as it does not contains $X = \frac{\partial}{\partial x_{1}}$ so the dimension of the Galois group of the formal variational equation is smaller than $5$. 
\end{proof}

Combining these two lemmas, we see that theorem \ref{main-theorem} now follows from theorem \ref{theoreme-casale-MR}.
\end{proof}

 \end{document}